\documentclass[11pt]{article}
\usepackage{mathrsfs}
\usepackage{amsmath}
\usepackage{amssymb}
\usepackage{amsthm}
\usepackage{epsfig}
\usepackage{graphics}
\usepackage{graphicx}
\usepackage{txfonts}
\usepackage{amsfonts}
\usepackage{multirow,tabularx}
\usepackage{caption}
\usepackage{float}
\usepackage{comment}
\usepackage{latexsym}
\usepackage{indentfirst}
\usepackage{url}
\usepackage{subfigure} 

\usepackage{color}

\newtheorem{theorem}{Theorem}[section]
\newtheorem{lemma}[theorem]{Lemma}
\newtheorem{corollary}[theorem]{Corollary}
\newtheorem{proposition}[theorem]{Proposition}

\theoremstyle{definition}
\newtheorem{definition}{Definition}[section]
\newtheorem{example}{Example}[section]

\theoremstyle{remark}
\newtheorem{remark}{Remark}[section]

\numberwithin{equation}{section}
\numberwithin{figure}{section}
\numberwithin{table}{section}

\def\RR{\mathbb{R}}
\def\Rd{\mathbb{R}^d}
\def\NN{\mathbb{N}}

\def\PP{\mathbb{P}}

\def\vx{\boldsymbol{x}}
\def\vy{\boldsymbol{y}}
\def\vc{\boldsymbol{c}}
\def\vz{\boldsymbol{z}}

\def\vxi{\boldsymbol{\xi}}

\def\vv{\boldsymbol{v}}
\def\vm{\boldsymbol{m}}
\def\vf{\boldsymbol{f}}
\def\vg{\boldsymbol{g}}
\def\vu{\boldsymbol{u}}

\def\ud{\textup{d}}
\def\me{\textup{e}}

\def\vP{\mathbf{P}}
\def\vQ{\mathbf{Q}}

\def\np{n_p}
\def\nb{n_b}

\newcommand{\abs}[1]{\left\lvert#1\right\rvert}
\newcommand{\norm}[1]{\lVert#1\rVert}

\def\Cont{\mathrm{C}}
\def\Leb{\mathrm{L}}
\def\Order{\mathcal{O}}

\def\diag{\mathrm{diag}}
\def\Span{\mathrm{span}}

\def\Hil{\mathcal{H}}
\def\Hilbert{\mathrm{H}}

\def\vA{\mathsf{A}}
\def\vV{\mathsf{V}}
\def\vB{\mathsf{B}}
\def\vPhi{\mathsf{\Phi}}
\def\vPsi{\mathsf{\Psi}}

\def\Filter{\mathcal{F}}
\def\Borel{\mathscr{B}}

\def\Normal{\mathcal{N}}

\def\Domain{\mathcal{D}}
\def\Eset{\mathcal{E}}

\def\vnoise{\boldsymbol{\xi}}

\def\Cov{\text{Cov}}
\def\Var{\text{Var}}

\def\erfc{\textup{erfc}}
\def\diag{\text{diag}}

\def\Kstar{\overset{*}{K}}

\def\vS{\boldsymbol{S}}
\def\vmu{\boldsymbol{\mu}}

\def\vK{\mathsf{K}}
\def\vKstar{\overset{*}{\mathsf{K}}}
\def\vk{\boldsymbol{k}}
\def\vkstar{\overset{*}{\boldsymbol{k}}}

\def\Aset{\mathcal{A}}

\DeclareMathOperator*{\argmin}{argmin}
\DeclareMathOperator*{\argmax}{argmax}

\begin{document}

\title{Approximation of Stochastic Partial Differential Equations by a Kernel-based Collocation Method}

\author{ Igor Cialenco, Gregory E. Fasshauer and Qi Ye\thanks{Corresponding author} \\
{\small Department of Applied Mathematics, Illinois Institute of Technology }\\
{\small 10 West 32nd Str, Bld E1, Room 208 }\\
{\small Chicago, Illinois, USA 60616}\\
{\small \mbox{igor@math.iit.edu}, \mbox{fasshauer@iit.edu} and  \mbox{qye3@iit.edu}}\\
{\small \emph{Updated Version in International Journal of Computer Mathematics}} \\
{\small \emph{Closed to Ye's Doctoral Thesis~\cite{YePhD2012}}}
}

\date{}

\maketitle

%---------------------------------------------------------------------------------------------------------------------
%/////////////////////////////////////////////////////////////////////////////////////////////////////////////////////
%---------------------------------------------------------------------------------------------------------------------

\begin{abstract}

In this paper we present the theoretical framework needed to justify the use of a kernel-based collocation method (meshfree approximation method) to estimate the solution of high-dimensional stochastic partial differential equations (SPDEs). Using an implicit time stepping scheme, we transform stochastic parabolic equations into stochastic elliptic equations. Our main attention is concentrated on the numerical solution of the elliptic equations at each time step. The estimator of the solution of the elliptic equations is given as a linear combination of reproducing kernels derived from the differential and boundary operators of the SPDE centered at collocation points to be chosen by the user. The random expansion coefficients are computed by solving a random system of linear equations. Numerical experiments demonstrate the feasibility of the method.



\end{abstract}

%---------------------------------------------------------------------------------------------------------------------
%/////////////////////////////////////////////////////////////////////////////////////////////////////////////////////
%---------------------------------------------------------------------------------------------------------------------

\section{Introduction}

Stochastic partial differential equations (SPDEs) frequently arise from applications in areas such as physics, engineering and finance. However, in many cases it is difficult to derive an explicit form of their solution. Moreover, current numerical algorithms often show limited success for high-dimensional problems or in complex domains -- even for deterministic partial differential equations. The \emph{kernel-based} approximation method (\emph{meshfree} approximation method~\cite{Buhmann2003,Fasshauer2007,Wendland2005}) is a relatively new numerical tool for the solutions of high-dimensional problems. In this paper we apply -- to our knowledge for the first time -- such a kernel-based collocation method to construct numerical estimators for stochastic partial differential equations. Galerkin-type approximation methods based on the eigenvalues and eigenfunctions of the underlying differential operator are currently very popular for the numerical solution of SPDEs \cite{DebBabuvskaOden2001,Muller-GronbachRitterWagner2008,JentzenKloeden2010a}. With the kernel-based meshfree collocation method introduced here explicit knowledge of these eigenvalues and eigenfunctions is not required since the kernels can be directly obtained as Green kernels of the differential operators \cite{FasshauerYe2011,FasshauerYe2011online}.  Stochastic collocation methods using a polynomial basis are also frequently found in the literature \cite{BabuvskaNobileTempone2010,NobileTemponeWebster2008}. These methods usually require the collocation points to lie on a regular grid. In our method the collocation points can be placed at rather arbitrarily scattered locations. This allows for the use of either deterministic or random designs such as, e.g., uniform or Sobol' points, but also for adaptively chosen locations. In this paper we do not study the design aspect of our algorithm, but reserve this important aspect for future work. Another advantage of using a meshfree method is its ability -- also not explicitly pursued here -- to deal with problems on a complex domain $\Domain\subset\Rd$, $d\geq1$, by using appropriately placed collocation points.
Another advantage of this method is its high efficiency, in the sense that once certain matrices are inverted and factored we can compute, essentially for free, the value of the approximated solution at any point in the spatial domain and at any event from sample space. In particular the method is suitable for simulation of a large number of sample paths of the solution.
In this article we present only a general framework for this new numerical method and prove weak convergence of the corresponding schemes. We conclude the paper with a numerical implementation of this method applied to a one-dimensional stochastic heat equation with Dirichlet boundary conditions driven by an additive space-time white noise (colored in space). Much more details, as well as some of the aspects just mentioned, will be discussed in Qi Ye's Ph.D. thesis~\cite{YePhD2012} or in future publications.

%---------------------------------------------------------------------------------------------------------------------

\subsection{The method in a nutshell}

Assume that $\Domain$ is a \emph{regular} open bounded domain in $\Rd$ (see Appendix~\ref{s:diff-bound}), and let $\Hil$ be a separable Hilbert space of functions defined on $\Domain$.
Also, let $(\Omega_W, \Filter_W,\{\Filter_t\}, \PP_W)$ be a stochastic basis with the usual assumptions.
We consider the following parabolic It\^{o} equation
\begin{equation}\label{eq:main}
\begin{cases}
\ud U_t=\mathcal{A} U_t \ud t+ \sigma \ud W_t,&\text{in } \Domain,\quad 0<t<T,\\
BU_t=0,&\text{on }\partial\Domain,\\
U_0=u_0,
\end{cases}
\end{equation}
where $\mathcal{A}$ is a linear elliptic operator in $\Hil$, $B$ is a boundary operator for Dirichlet or Neumann boundary conditions, $u_0\in\Hil$, and $W$ is a Wiener process in $\Hil$, with mean zero and spatial covariance function $R$ given by $\mathbb{E}(W(t,\vx)W(s,\vy))= \min\{t,s\}R(\vx,\vy), \ \vx,\vy\in\Domain, \ t,s>0$, and $\sigma>0$ (see for instance~\cite{Chow2007}).

We assume that equation \eqref{eq:main} has a unique solution $U\in\Leb_2(\Omega_W\times(0,T); \Hil)$.

The proposed numerical method for solving a general SPDE of the form \eqref{eq:main} can be described as follows:
\begin{enumerate}

\item[S1)] Discretize \eqref{eq:main} in time by the implicit Euler scheme
  at equally spaced time points $0=t_0<t_1<\ldots <t_n=T$,
\begin{equation}\label{eq:mainDiscrete}
  U_{t_j} - U_{t_{j-1}} = \mathcal{A}U_{t_j} \delta t + \sigma \delta W_{j}, \quad j=1,\ldots,n,
\end{equation}
where $\delta t:=t_{j}-t_{j-1}$ and $\delta W_{j} := W_{t_{j}} - W_{t_{j-1}}$.

\item[S2)] Since it follows from \eqref{eq:mainDiscrete} and the definition of Brownian motion that the noise increment $\delta W_j$ at each time instance $t_j$ is independent from the solution $U_{t_{j-1}}$ at the previous step, we simulate the Gaussian field with covariance structure $R(\vx,\vy)$ at a finite collection of predetermined \emph{collocation points}
\[
X_{\Domain}:=\left\{\vx_1,\cdots,\vx_N\right\}\subset\Domain,\quad
X_{\partial\Domain}:=\left\{\vx_{N+1},\cdots,\vx_{N+M}\right\}\subset\partial\Domain.
\]

\item[S3)] Let the differential operator $P :=I-\delta t\mathcal{A}$, and the  noise term $\xi:=\sigma \delta W_j$. Since $\xi$ is a Gaussian field with
$\mathbb{E}(\xi_{\vx})=0$ and $\Cov(\xi_{\vx}\xi_{\vy})=\sigma^2\delta tR(\vx,\vy)$, equation \eqref{eq:mainDiscrete} together with the corresponding boundary condition becomes an elliptic SPDE of the form
\begin{equation}\label{eq:S3}
\begin{cases}
Pu=f+\xi,&\text{in }\Domain,\\
Bu=0,&\text{on }\partial\Domain,
\end{cases}
\end{equation}
where $u:=U_{t_j}$ is seen as an unknown part and $f:=U_{t_{j-1}}$ and $\xi$ are viewed as given parts.
We will solve for $u$ using a kernel-based collocation method written as
\begin{equation}\label{e:kernel_expansion}
u(\vx)\approx\hat{u}(\vx):=
\sum_{k=1}^Nc_{k}P_2\Kstar(\vx,\vx_k)+\sum_{k=1}^Mc_{N+k}B_2\Kstar(\vx,\vx_{N+k}),\quad \vx\in\Domain,
\end{equation}
where $K$ is a \emph{reproducing kernel} and the integral-type kernels $\Kstar, P_2\Kstar, B_2\Kstar$ are defined in Lemmas~\ref{l:Gauss-RK} and~\ref{l:P-B-Cov-expan}. The unknown random coefficients $\vc:=\left(c_1,\cdots,c_{N+M}\right)^T$ are obtained by solving a random system of linear equations (with constant deterministic system matrix and different random right-hand side) at each time step.
Details are provided in Section \ref{s:COL-SPDE}.

\item[S4)] Repeat S2 and S3 for all $j=1,\ldots,n$.

\end{enumerate}

Obviously, many other -- potentially better -- time stepping schemes could be applied here. However, as mentioned earlier, we focus mainly on step S3 and are for the time being content with using the implicit Euler scheme.
Naturally, the rate of convergence of the above numerical scheme depends on the size of the time step $\delta t$, and the \emph{fill distance} $h_X:=\sup_{\vx\in\Domain}\min_{\vx_k\in X_{\Domain}\cup X_{\partial\Domain}}\norm{\vx-\vx_k}_2$ of the collocation points.
%the \emph{fill distance} of the collocation points,
%\[
%h_X:=\sup_{\vx\in\Domain}\min_{\vx_k\in X_{\Domain}\cup X_{\partial\Domain}}\norm{\vx-\vx_k}_2.
%\]
We support this statement empirically in Section~\ref{sec:numerical-experimants}. We should mention that even for deterministic time-dependent PDEs to find the exact rates of convergence of kernel-based methods is a delicate and nontrivial question, only recently solved in \cite{HonSchaback2010}. We will address this question in the case of SPDEs in future works.

The fundamental building block of our mesh-free method is the reproducing kernel $K:\Domain\times\Domain\rightarrow\RR$ and its reproducing-kernel Hilbert space $\Hilbert_K(\Domain)$ (see Appendix~\ref{s:RKHS} for more details). By the very nature of such a kernel-based method, the approximate solution $U_{t_j}, \ j=1,\ldots,n$, must live in $\Hilbert_K(\Domain)$. Thus, we make the following standing assumption throughout the paper:
\begin{itemize}
\item[] The solution $U$ of the original equation \eqref{eq:main} lies in a Hilbert space $\mathcal{H}$ which can be embedded in the reproducing kernel Hilbert space $\Hilbert_K(\Domain)$.
\end{itemize}
Usually, $\mathcal{H}$ is a Sobolev space $\Hil^m(\Domain)$, for some positive $m$. In this case it is possible to choose an appropriate kernel $K$ such that the above embedding assumption is satisfied. For a general discussion of existence and uniqueness of the solution of problem \eqref{eq:main} see, e.g., \cite{RozovskiiBook,DaPrato,Chow2007}.

\section{Reproducing-kernel collocation method for Gaussian processes}\label{RK-COL-GP}

In this section we briefly review the standard kernel-based approximation method for high-dimensional interpolation problems. However, since we will later be interested in solving a stochastic PDE, we present the following material mostly from the stochastic point of view. For further discussion of this method we refer the reader to the recent survey papers \cite{ScheuererSchabackSchlather2010,Fasshauer2011} and references therein.

Assume that the function space $\Hilbert_K(\Domain)$ is a reproducing-kernel Hilbert space with norm $\| \cdot \|_{K,\Domain}$ and its reproducing kernel $K\in\Cont(\overline{\Domain}\times\overline{\Domain})$ is symmetric positive definite (see Appendix~\ref{d:RKHS}). Given the data values $\{y_j\}_{j=1}^N\subset\RR$ at the collocation points $X_{\Domain}:=\{\vx_j\}_{j=1}^N\subset\Domain$ of an unknown function $u\in\Hilbert_K(\Domain)$, i.e.,
\[
y_j=u(\vx_j),\quad \vx_j=(x_{1,j},\cdots,x_{d,j})\in\Domain\subset\Rd,\quad j=1,\ldots,N,
\]
the goal is to find \textit{an optimal estimator} from $\Hilbert_K(\Domain)$ that interpolates these data.

%////////////////////////////////////////////////////////////////////////////////////////////////////////////////////////
\begin{definition}[{\cite[Definition~3.28]{BerlinetThomas2004}}]\label{d:Gaussian}
A stochastic process $S:\Domain\times\Omega\rightarrow\RR$ is said to be \emph{Gaussian} with mean $\mu:\Domain\rightarrow\RR$
and covariance kernel $\Phi:\Domain\times\Domain\rightarrow\RR$ on a probability space $(\Omega,\Filter,\PP)$ if,
for any pairwise distinct points $X_{\Domain}:=\{\vx_1,\cdots,\vx_N\}\subset\Domain$, the random vector $\vS:=(S_{\vx_1},\cdots,S_{\vx_N})^T$
is a multi-normal random variable on $(\Omega,\Filter,\PP)$ with mean $\vmu$ and covariance matrix $\vPhi$, i.e.,
$\vS\sim\Normal(\vmu,\vPhi)$,
where $\vmu:=(\mu(\vx_1),\cdots,\mu(\vx_N))^T$ and $\vPhi:=(\Phi(\vx_j,\vx_k))_{j,k=1}^{N,N}$.
\end{definition}
%////////////////////////////////////////////////////////////////////////////////////////////////////////////////////////

%------------------------------------------------------------------------------------------------------------------------
%------------------------------------------------------------------------------------------------------------------------

\subsection{Data fitting problems via deterministic interpolation and simple kriging}

In the \emph{deterministic formulation} of kernel interpolation we solve an optimization problem by minimizing the reproducing-kernel norm subject to interpolation constraints, i.e.,
\[
\hat{u}_K = \argmin_{u\in\Hilbert_K(\Domain)}\left\{\norm{u}_{K,\Domain} \text{ s.t. } u(\vx_j)=y_j,\ j=1,\ldots,N\right\}.
\]
%\[
%\begin{split}
%&\min_{f\in\Hilbert_K(\Domain)}\norm{f}_{K,\Domain}\\
%\text{s.t. }&f(\vx_j)=y_j,\quad j=1,\ldots,N.
%\end{split}
%\]
In this case, the minimum norm interpolant (also called the collocation solution) $\hat{u}_K(\vx)$ is a linear combination of ``shifts'' of the reproducing kernel $K$,
\begin{equation}\label{e:classical-intrp}
\hat{u}_K(\vx):=\sum_{k=1}^Nc_kK(\vx,\vx_k),\quad \vx\in\Domain ,
\end{equation}
where the coefficients $\vc:=(c_1,\cdots,c_N)^T$ are obtained by solving the following system of linear equations
\begin{equation}\label{eq:DetProblemCs}
\vK\vc=\vy_0,
\end{equation}
with $\vK:=\left(K(\vx_j,\vx_k)\right)_{j,k=1}^{N,N}$ and $\vy_0:=(y_1,\cdots,y_N)^T$.

For simple kriging, i.e., in the \emph{stochastic formulation}, we let $S$ be a Gaussian process with mean $0$ and covariance kernel $K$ on some probability space $(\Omega,\Filter,\PP)$.
Kriging is based on the modeling assumption that $u$ is a realization of the Gaussian field $S$. The data values $y_1,\ldots,y_N$ are then realizations of the random variables $S_{\vx_1},\ldots,S_{\vx_N}$. The optimal unbiased predictor of $S_{\vx}$ based on $\vS$ is equal to
\[
\hat{U}_{\vx}:=\sum_{k=1}^Nc_k(\vx)S_{\vx_k}=
\argmin_{U\in\Span\{S_{\vx_j}\}_{j=1}^N}\mathbb{E}\abs{U-S_{\vx}}^2,
%\underset{U\in\Span\{S_{\vx_j}\}_{j=1}^N}{\arg\min}\mathbb{E}\abs{U-S_{\vx}}^2,
\]
where the coefficients $\vc(\vx):=\left(c_1(\vx),\cdots,c_N(\vx)\right)^T$ are given by
\[
\vc(\vx)=\vK^{-1}\vk(\vx)
\]
with $\vk(\vx):=\left(K(\vx,\vx_1),\cdots,K(\vx,\vx_N)\right)^T$ and the same matrix $\vK$ as above.
We can also compute that
\[
\mathbb{E}(\hat{U}_{\vx}|S_{\vx_1}=y_1,\cdots,S_{\vx_N}=y_N)=\hat{u}_{K}(\vx).
\]

Note that, in the kriging approach we consider only the values of the stochastic process $S$ at the collocation points, and view the obtained vector as a random variable.
However, if we view $S$ as a real function, then
$\PP(S\in\Hilbert_K(\Domain))=0$ by \cite[Theorem~7.3]{LukicBeder2001}. A simple example for this fact is given by the scalar Brownian motion defined in the domain $\Domain:=(0,1)$ (see, e.g., \cite[Example~5.1]{FasshauerYe2011}). This means that it is difficult to apply the kriging formulation to PDE problems. Next we will introduce a new stochastic data fitting approach that will subsequently allow us to perform kernel-based collocation for stochastic PDEs.

%------------------------------------------------------------------------------------------------------------------------
%------------------------------------------------------------------------------------------------------------------------

\subsection{Data fitting problems via a new stochastic approach} \label{sec:StochasticProblems}

From now on we will view the separable reproducing-kernel Hilbert space $\Hilbert_K(\Domain)$ as a sample space and its Borel $\sigma$-field $\Borel(\Hilbert_K(\Domain))$ as a $\sigma$-algebra to set up the probability spaces so that the stochastic process $S_{\vx}(\omega):=\omega(\vx)$ is Gaussian.
We use the techniques of~\cite{Janson1997,LukicBeder2001} to verify Lemma~\ref{l:Gauss-RK}, which is a restatement of \cite[Theorem~7.2]{LukicBeder2001}.
This theoretical result is a generalized form of Wiener measure defined on the measurable space $(\Cont[0,\infty),\Borel(\Cont[0,\infty)))$, called canonical space, such that the coordinate mapping process $W_x(\omega):=\omega(x)$ is a Brownian motion (see, for instance, \cite{KaratzasShreve1991},~Chapter~2).

%////////////////////////////////////////////////////////////////////////////////////////////////////////////////////////
\begin{lemma}\label{l:Gauss-RK}
Let the positive definite kernel $K\in\Cont(\overline{\Domain}\times\overline{\Domain})$ be the reproducing kernel of the reproducing-kernel Hilbert space $\Hilbert_K(\Domain)$.
Given a function $\mu\in\Hilbert_K(\Domain)$ there exists a probability measure $\PP^{\mu}$ defined on $(\Omega_K,\Filter_K):=(\Hilbert_K(\Domain),\Borel(\Hilbert_K(\Domain)))$
such that $S_{\vx}(\omega):=\omega(\vx)$ is Gaussian
with mean $\mu$ and covariance kernel $\Kstar$ on $(\Omega_K,\Filter_K,\PP^{\mu})$, where
the integral-type kernel $\Kstar$ of $K$ is given by
\[
\Kstar(\vx,\vy):=\int_{\Domain}K(\vx,\vz)K(\vy,\vz)\ud\vz,\quad \vx,\vy\in\Domain.
\]
Moreover, the process $S$ has the following expansion
\[
S_{\vx}=\sum_{k=1}^{\infty}\zeta_k\sqrt{\lambda_k}e_k(\vx),\quad \vx\in\Domain,\quad \PP^{\mu}\text{-a.s.},
\]
where $\{\lambda_k\}_{k=1}^{\infty}$ and $\{e_k\}_{k=1}^{\infty}$ are the eigenvalues and eigenfunctions of the reproducing kernel $K$,
and $\zeta_k$ are independent Gaussian random variables with mean $\hat{\mu}_k:=\langle \mu,\sqrt{\lambda_k}e_k \rangle_{K,\Domain}$
and variance $\lambda_k$, $k\in\NN$.
\end{lemma}

Before we prove Lemma~\ref{l:Gauss-RK} we remark that we have introduced the integral-type kernel $\Kstar$ for convenience only. As seen later, in order to ``match the spaces'', any other kernel that ``dominates'' $K$ (in the sense of \cite{LukicBeder2001}) could play the role of the integral-type kernel $\Kstar$.

%////////////////////////////////////////////////////////////////////////////////////////////////////////////////////////
\begin{proof}
We first consider the case when $\mu=0$. There exist countably many independent standard normal random variables $\{\xi_k\}_{k=1}^{\infty}$ on a probability space $(\Omega_{\xi},\Filter_{\xi},\PP_{\xi})$,
i.e., $\xi_k\sim i.i.d.~\Normal(0,1)$, $k\in\NN$.
Let $\{\lambda_k\}_{k=1}^{\infty}$ and $\{e_k\}_{k=1}^{\infty}$ be the eigenvalues and eigenfunctions of the reproducing kernel $K$ as in Theorem~\ref{t:RKHS}. We define $S:=\sum_{k=1}^{\infty}\xi_k\lambda_ke_k$ $\PP_{\xi}$-a.s. Note that $S$ is Gaussian with mean $0$ and covariance kernel $\Kstar$. Since $\mathbb{E}(\sum_{k=1}^{\infty}\xi_k^2\lambda_k)\leq\sum_{k=1}^{\infty}\Var(\xi_k)\lambda_k=\sum_{k=1}^{\infty}\lambda_k<\infty$ indicates that $\sum_{k=1}^{\infty}\abs{\xi_k\sqrt{\lambda_k}}^2<\infty$ $\PP_{\xi}$-a.s., Theorem~\ref{t:RKHS} shows that $S(\cdot,\omega)\in\Hilbert_K(\Domain)$ $\PP_{\xi}$-a.s. Therefore $S$ is a measurable map from $(\Omega_{\xi},\Filter_{\xi})$ into $(\Omega_K,\Filter_K)$ by \cite[Chapter 4.3.1]{BerlinetThomas2004} and \cite[Lemma~2.1]{LukicBeder2001}. On the other hand, the probability measure $\PP^{0}:=\PP_{\xi}\circ S^{-1}$ (also called the \emph{image measure} of $\PP_{\xi}$ under $S$) is well defined on $(\Omega_K,\Filter_K)$, i.e., $\PP^{0}(A):=\PP_{\xi}(S^{-1}(A))$ for each $A\in\Filter_K$. Hence, $S$ is also a Gaussian process with mean $0$ and covariance kernel $\Kstar$ on $(\Omega_K,\Filter_K,\PP^{0})$.

Let $S^{\mu}:=S+\mu$ on $(\Omega_K,\Filter_K,\PP^{0})$. Then $\mathbb{E}(S^{\mu}_{\vx})=\mathbb{E}(S_{\vx})+\mu(\vx)$ and $\Cov(S^{\mu}_{\vx},S^{\mu}_{\vy})=\Cov(S_{\vx},S_{\vy})$ with respect to $\PP^{0}$. We define a new probability measure $\PP^{\mu}$ by $\PP^{\mu}(A):=\PP^{0}(A-\mu)$ for each $A\in\Filter_K$. It is easy to check that $\Omega_K+\mu=\Hilbert_K(\Omega)=\Omega_K$ and
$\{\mu+A:A\in\Filter_K\}=\Borel(\Hilbert_K(\Omega))=\Filter_K$. Thus $S$ is Gaussian with mean $\mu$ and covariance kernel $\Kstar$ on $(\Omega_K,\Filter_K,\PP^{\mu})$.

Moreover, since $\mu\in\Hilbert_K(\Omega)$, it can be expanded in the form $\mu=\sum_{k=1}^{\infty}\hat{\mu}_k\sqrt{\lambda_k}e_k$, where $\hat{\mu}_k:=\langle \mu,\sqrt{\lambda_k}e_k \rangle_{K,\Domain}$, so that
$S^{\mu}=\sum_{k=1}^{\infty}(\hat{\mu}_k+\sqrt{\lambda_k}\xi_k)\sqrt{\lambda_k}e_k$. But then $\zeta_k\sim\hat{\mu}_k+\sqrt{\lambda_k}\xi_k\sim\Normal(\hat{\mu}_k,\lambda_k)$ are independent  on $(\Omega_K,\Filter_K,\PP^{\mu})$.
\end{proof}
%////////////////////////////////////////////////////////////////////////////////////////////////////////////////////////

According to \cite[Theorem~4.91]{BerlinetThomas2004}, we can also verify that the random variable $V_f(\omega):=\langle \omega,f \rangle_{K,\Domain},\  f\in\Hilbert_K(\Domain)$,
is a scalar normal variable on $(\Omega_K,\Filter_K,\PP^{\mu})$, i.e.,
\[
V_f(\omega)\sim\Normal(m_f,\sigma_f^2),\quad \omega\in\Omega_K=\Hilbert_K(\Domain),
\]
where $m_f:=\langle \mu,f \rangle_{K,\Domain}$ and $\sigma_f:=\norm{f}_{\Leb_2(\Domain)}$. Therefore the probability measure $\PP^{\mu}$ defined in Lemma~\ref{l:Gauss-RK} is \emph{Gaussian}.

\bigskip

Let $p_X^{\mu}:\RR^N\rightarrow\RR$ be the joint probability density function of $S_{\vx_1},\cdots,S_{\vx_N}$ defined on $(\Omega_K,\Filter_K,\PP^{\mu})$. Then it is a normal density function with mean
$\vmu:=(\mu(\vx_1),\cdots,\mu(\vx_N))^T$ and covariance matrix $\vKstar:=(\Kstar(\vx_j,\vx_k))_{j,k=1}^{N,N}$. In analogy to the kriging formulation we can find the optimal mean function $\hat{\mu}\in\Hilbert_K(\Domain)$ fitting the data values $\vy_0:=(y_1,\cdots,y_N)^T$, i.e.,
\[
\hat{\mu}:=\vkstar{}^{T}\vKstar{}^{-1}\vy_0=\underset{\mu\in\Hilbert_K(\Domain)}{\sup}p_X^{\mu}(\vy_0)
=\underset{\mu\in\Hilbert_K(\Domain)}{\sup}\PP^{\mu}(\vS=\vy_0),
\]
where $\vkstar(\vx):=(\Kstar(\vx,\vx_1),\cdots,\Kstar(\vx,\vx_N))^T$.

We now fix any $\vx\in\Domain$. Straightforward calculation shows that the random variable $S_{\vx}$,
given $S_{\vx_1},\cdots,S_{\vx_N}$, defined on $(\Omega_K,\Filter_K,\PP^{\mu})$ has a conditional probability density function
\[
p_{\vx}^{\mu}(v|\vv):=\frac{1}{\sigma(\vx)\sqrt{2\pi}}\exp\left(-\frac{(v-m^{\mu}_{\vx}(\vv))^2}{2\sigma(\vx)^2}\right),\quad v\in\RR, \ \vv\in\RR^N,
\]
where $m^{\mu}_{\vx}(\vv):=\mu(\vx)+\vkstar(\vx)^T\vKstar{}^{-1}(\vv-\vm^{\mu})$,
$\vm^{\mu}:=(\mu(\vx_1),\cdots,\mu(\vx_N))^T$,
and $\sigma(\vx)^2:=\Kstar(\vx,\vx)-\vkstar(\vx)^{T}\vKstar{}^{-1}\vkstar(\vx)$.
%%%
Then the optimal estimator that maximizes the probability
\[
\begin{split}
&\max_{v\in\RR}\PP^{\hat{\mu}}\left(\left\{\omega\in\Omega_K:\omega(\vx)=v\text{ s.t. }\omega(\vx_1)=y_1,\cdots,\omega(\vx_N)=y_N\right\}\right)\\
=&\max_{v\in\RR}\PP^{\hat{\mu}}\left(S_{\vx}=v\big| S_{\vx_1}=y_1,\cdots,S_{\vx_N}=y_N\right)\\
\end{split}
\]
is given by
\[
\hat{u}(\vx):=\hat{\mu}(\vx)=\argmax_{v\in\RR} \ p_{\vx}^{\hat{\mu}}(v|\vy_0).
\]

\begin{proposition}\label{prop:main-approximation}
With the above notations, the following  equality holds true
\begin{equation}
p_{\vx}^{\hat{\mu}}(\hat{u}(\vx)|\vy_0)  =\sup_{v\in\RR,\mu\in\Hilbert_K(\Domain)}p_{\vx}^{\mu}(v|\vy_0). \label{eq:Our-MLE}
\end{equation}
Moreover, for any $\epsilon>0$,
\begin{equation}
\sup_{\mu\in\Hilbert_K(\Domain)}\PP^{\mu}(\abs{\hat{u}(\vx)-u(\vx)}\geq\epsilon) \leq
\sup_{\mu\in\Hilbert_K(\Domain)}\PP^{\mu}(\Eset_{\vx}^{\epsilon})
=\erfc\left(\frac{\epsilon}{\sqrt{2}\sigma(\vx)}\right), \label{eq:weak-error-bound}
\end{equation}
where
\[
\Eset_{\vx}^{\epsilon}:=\left\{\omega\in\Omega_K:
\abs{\omega(\vx)-\hat{u}(\vx)}\geq\epsilon\text{ s.t. }\omega(\vx_1)=y_1,\cdots,\omega(\vx_N)=y_N
\right\}.
\]
\end{proposition}
Identity \eqref{eq:Our-MLE} follows by direct evaluations.  Consequently, taking into account that $S$ is Gaussian, inequality \eqref{eq:weak-error-bound} follows also immediately.
\begin{remark}
Instead of giving a deterministic (or strong) error bound for the proposed numerical scheme, we provide \emph{a weak type convergence} of the approximated solution $\hat{u}$ to the true solution $u$, as stated in Proposition \ref{prop:main-approximation}. In fact, inequality \eqref{eq:weak-error-bound} can be seen as a confidence interval for the estimator $\hat{u}$ with respect to the probability measure $\mathbb{P}^\mu$.
\end{remark}

In the next section we generalize this stochastic approach to solve elliptic PDEs.

\section{Collocation Method for Elliptic PDEs and SPDEs}\label{s:COL-SPDE}

We begin by setting up Gaussian processes via reproducing kernels with differential and boundary operators.

Suppose that the reproducing-kernel Hilbert space $\Hilbert_K(\Domain)$ is embedded into the Sobolev space $\Hil^{m}(\Domain)$ where $m>d/2$.
Let $n:=\lceil m -d/2\rceil-1$. By the Sobolev embedding theorem
%(Rellich-Kondrachov theorem) \ig{it is really Sobolev. For p=2 it is even simpler case. I would delete Rellich-Kondrachov.}
$\Hil^{m}(\Domain)\subset\Cont^{n}(\overline{\Domain})$.
When the differential operator $P$ and the boundary operator $B$ have the orders $\Order(P)<m-d/2$ and $\Order(B)<m-d/2$, then the stochastic processes $PS_{\vx}(\omega):=(P\omega)(\vx)$ and $BS_{\vx}(\omega):=(B\omega)(\vx)$ are well-defined on $(\Omega_K,\Filter_K,\PP^{\mu})$. According to Lemma~\ref{l:P-B-expan}, we have
$PS=\sum_{k=1}^{\infty}\zeta_k\sqrt{\lambda_k}Pe_k$ and $BS=\sum_{k=1}^{\infty}\zeta_k\sqrt{\lambda_k}Be_k$. If $\mu\in\Hilbert_K(\Domain)\subseteq\Hil^m(\Domain)$, then $P\mu\in\Cont(\overline{\Domain})$ and $B\mu\in\Cont(\partial\Domain)$. Lemma~\ref{l:P-B-Cov-expan} implies that $P_1P_2\Kstar(\vx,\vy)=\sum_{k=1}^{\infty}\lambda_k^2Pe_k(\vx)Pe_k(\vy)$ and $B_1B_2\Kstar(\vx,\vy)=\sum_{k=1}^{\infty}\lambda_k^2Be_k(\vx)Be_k(\vy)$
(here we can use the fact that $\Cov(PS_{\vx},PS_{\vy})=P_{\vx}P_{\vy}\Cov(S_{\vx},S_{\vy})$
and $\Cov(BS_{\vx},BS_{\vy})=B_{\vx}B_{\vy}\Cov(S_{\vx},S_{\vy})$).
Applying Lemma~\ref{l:Gauss-RK}, we can obtain the main theorem for the construction of Gaussian processes via  reproducing kernels coupled with differential or boundary operators.
%////////////////////////////////////////////////////////////////////////////////////////////////////////////////////////
\begin{theorem}\label{t:Gauss-RK-PB}
Suppose that the reproducing kernel Hilbert space $\Hilbert_K(\Domain)$ is embedded into the Sobolev space $\Hil^{m}(\Domain)$ with $m>d/2$. Further assume that the differential operator $P$ and the boundary operator $B$ have the orders $\Order(P)< m-d/2$ and $\Order(B)< m-d/2$.
Given a function $\mu\in\Hilbert_K(\Domain)$ there exists a probability measure $\PP^{\mu}$ defined on $(\Omega_K,\Filter_K)=(\Hilbert_K(\Domain),\Borel(\Hilbert_K(\Domain)))$ (as in Lemma~\ref{l:Gauss-RK}) such that the
stochastic processes $PS$, $BS$ given by
\[
\begin{split}
&PS_{\vx}(\omega)=PS(\vx,\omega):=(P\omega)(\vx),\quad \vx\in\Domain\subset\Rd,\quad \omega\in\Omega_K=\Hilbert_K(\Domain),\\
&BS_{\vx}(\omega)=BS(\vx,\omega):=(B\omega)(\vx),\quad \vx\in\partial\Domain,\quad \omega\in\Omega_K=\Hilbert_K(\Domain),
\end{split}
\]
are jointly Gaussian processes with means $P\mu$, $B\mu$ and covariance kernels $P_1P_2\Kstar$, $B_1B_2\Kstar$ defined on $(\Omega_K,\Filter_K,\PP^{\mu})$, respectively.
In particular, they can be expanded as
\[
PS_{\vx}=\sum_{k=1}^{\infty}\zeta_k\sqrt{\lambda_k}Pe_k(\vx),~\vx\in\Domain,\text{ and }
BS_{\vx}=\sum_{k=1}^{\infty}\zeta_k\sqrt{\lambda_k}Be_k(\vx),~\vx\in\partial\Domain,~\PP^{\mu}\text{-a.s.},
\]
where $\{\lambda_k\}_{k=1}^{\infty}$ and $\{e_k\}_{k=1}^{\infty}$ are the eigenvalues and eigenfunctions of the reproducing kernel $K$ and their related Fourier coefficients are the independent normal variables $\zeta_k\sim\Normal(\hat{\mu}_k,\lambda_k)$ and $\hat{\mu}_k:=\langle \mu,\sqrt{\lambda_k}e_k \rangle_{K,\Domain}$, $k\in\NN$.
\end{theorem}
%////////////////////////////////////////////////////////////////////////////////////////////////////////////////////////

%////////////////////////////////////////////////////////////////////////////////////////////////////////////////////////
\begin{corollary}\label{c:cov-matrix-Gauss-RK-PB}
Suppose all notations and conditions are as in Theorem~\ref{t:Gauss-RK-PB}. Given collocation points
$X_{\Domain}:=\{\vx_j\}_{j=1}^N\subset\Domain$ and $X_{\partial\Domain}:=\{\vx_{N+j}\}_{j=1}^M\subset\partial\Domain$, the random vector
\[
\vS_{PB}:=(PS_{\vx_1},\cdots,PS_{\vx_N},BS_{\vx_{N+1}},\cdots,BS_{\vx_{N+M}})^T
\]
defined on $(\Omega_K,\Filter_K,\PP^{\mu})$ 
has a multi-normal distribution with mean $\vm^{\mu}_{PB}$ and covariance matrix $\vKstar_{PB}$, i.e.,
\[
\vS_{PB}\sim\Normal(\vm^{\mu}_{PB},\vKstar_{PB}),
\]
where $\vm^{\mu}_{PB}:=(P\mu(\vx_1),\cdots,P\mu(\vx_N),B\mu(\vx_{N+1}),\cdots,B\mu(\vx_{N+M}))^T$ and
\[
\vKstar_{PB}:=
\begin{pmatrix}
(P_1P_2\Kstar(\vx_j,\vx_k))_{j,k=1}^{N,N}, & (P_1B_2\Kstar(\vx_j,\vx_{N+k}))_{j,k=1}^{N,M}\\
(B_1P_2\Kstar(\vx_{N+j},\vx_k))_{j,k=1}^{M,N}, & (B_1B_2\Kstar(\vx_{N+j},\vx_{N+k}))_{j,k=1}^{M,M}\\
\end{pmatrix}.
\]
\end{corollary}
%////////////////////////////////////////////////////////////////////////////////////////////////////////////////////////
%
%////////////////////////////////////////////////////////////////////////////////////////////////////////////////////////
\begin{remark}
%The covariance matrix $\vKstar_{PB}$ may not be positive-definite but it is always positive semi-definite. Thus $\vKstar_{PB}$ has a pseudo-inverse $\vKstar_{PB}{}^\dag$ even when it is singular.
%\gf{I suggest the following instead:}
While the covariance matrix $\vKstar_{PB}$ may be singular, it is always positive semi-definite and therefore always has a pseudo-inverse $\vKstar_{PB}{}^\dag$.
\end{remark}
%////////////////////////////////////////////////////////////////////////////////////////////////////////////////////////

Using Corollary~\ref{c:cov-matrix-Gauss-RK-PB}, we can compute the joint probability density function $p_X^{\mu}$
of $\vS_{PB}$ defined on $(\Omega_K,\Filter_K,\PP^{\mu})$. In the same way, we can also get the joint density function $p_J^{\mu}$ of $(S_{\vx}, \vS_{PB})$ defined on $(\Omega_K,\Filter_K,\PP^{\mu})$.
By Bayes' rule, we can obtain the conditional probability
density function of the random variable $S_{\vx}$ given $\vS_{PB}$.
%////////////////////////////////////////////////////////////////////////////////////////////////////////////////////////
\begin{corollary}\label{c:cond-prob-density}
We follow the notations of Corollary~\ref{c:cov-matrix-Gauss-RK-PB}. For any fixed $\vx\in\Domain$, the random variable $S_{\vx}$ given $\vS_{PB}$ defined on $(\Omega_K,\Filter_K,\PP^{\mu})$ has a conditional probability density function
\[
p^{\mu}_{\vx}(v|\vv):=\frac{p_{J}^{\mu}(v,\vv)}{p_X^{\mu}(\vv)}=\frac{1}{\sigma(\vx)\sqrt{2\pi}}\exp\left(-\frac{(v-m^{\mu}_{\vx}(\vv))^2}{2\sigma(\vx)^2}\right),
\quad v\in\RR, ~\vv\in\RR^{N+M},
\]
where
\begin{align*}
m^{\mu}_{\vx}(\vv)&:=\mu(\vx)+\vk_{PB}(\vx)^T
\vKstar_{PB}{}^{\dag}(\vv-\vm^{\mu}_{PB}), \\
\sigma(\vx)^2 &:=\Kstar(\vx,\vx)-\vk_{PB}(\vx)^T \vKstar_{PB}{}^{\dag}\vk_{PB}(\vx),\\
\vk_{PB}(\vx)&:= (P_2\Kstar(\vx,\vx_1),\cdots,P_2\Kstar(\vx,\vx_N),B_2\Kstar(\vx,\vx_{N+1}),\cdots,B_2\Kstar(\vx,\vx_{N+M}))^T.
\end{align*}

In particular, given the real observation $\vy:=(y_1,\cdots,y_{N+M})^T$,  $S_{\vx}$ conditioned on $\vS_{PB}=\vy$ has the probability density $p_{\vx}^{\mu}(\cdot|\vy)$.
\end{corollary}
%////////////////////////////////////////////////////////////////////////////////////////////////////////////////////////
This corollary is similar to the features of Gaussian conditional distributions (see~\cite[Theorem~9.9]{Janson1997}).

%------------------------------------------------------------------------------------------------------------------------
%------------------------------------------------------------------------------------------------------------------------

\subsection{Elliptic deterministic PDEs}\label{s:Int-PDE}

Suppose that $u\in\Hilbert_K(\Domain)$ is the unique solution of the deterministic elliptic PDE
\begin{equation}\label{e:PDE-ell-no-noise}
\begin{cases}
Pu=f,&\text{in }\Domain\subset\Rd,\\
Bu=g,&\text{on }\partial\Domain,
\end{cases}
\end{equation}
where $f:\Domain\rightarrow\RR$ and $g:\partial\Domain\rightarrow\RR$.
Denote by $\{y_j\}_{j=1}^N$ and $\{y_{N+k}\}_{k=1}^M$ the values of $f$ and $g$ at the collocation points $X_{\Domain}$ and $X_{\partial\Domain}$, respectively:
\[
y_j:=f(\vx_j),\quad y_{N+k}:=g(\vx_{N+k}),\quad j=1,\cdots,N,\ k=1,\cdots,M.
\]

From now on we assume that the covariance matrix $\vKstar_{PB}$ defined in Corollary~\ref{c:cov-matrix-Gauss-RK-PB} is nonsingular and we therefore can replace pseudo-inverses with inverses.

Let $\vy_0:=(y_1,\cdots,y_N,y_{N+1},\cdots,y_{N+M})^T$, and denote by $p^{\mu}_{\vx}(\cdot|\cdot)$ the conditional density function defined in Corollary~\ref{c:cond-prob-density}. We approximate the solution $u$ of \eqref{e:PDE-ell-no-noise} by the optimal estimator $\hat{u}(\vx)$ derived in the previous section, i.e., we maximize the conditional probability given the data values $\vy_0$:
\[
u(\vx)\approx\hat{u}(\vx)=\argmax_{v\in\RR}\sup_{\mu\in\Hilbert_K(\Domain)}p^{(\mu)}_{\vx}(v|\vy_0),
\quad \vx\in\Domain.
\]
By direct evaluation as in Section~\ref{sec:StochasticProblems} one finds that
\[
\hat{u}(\vx):=\vk_{PB}(\vx)^T\vKstar_{PB}{}^{-1}\vy_0,\quad \vx\in\Domain,
\]
where the basis functions $\vk_{PB}(\vx)$ are defined in Corollary~\ref{c:cond-prob-density}. Moreover, the estimator $\hat{u}\in\Hilbert_K(\Domain)$ fits all the data values: $P\hat{u}(\vx_1)=y_1,\ldots,P\hat{u}(\vx_N)=y_N$ and $B\hat{u}(\vx_{N+1})=y_{N+1},\ldots,B\hat{u}(\vx_{N+M})=y_{N+M}$. This means that we have computed a collocation solution of the PDE \eqref{e:PDE-ell-no-noise}.
Also note that $\hat{u}$ can be written as a linear combination of the kernels as in \eqref{e:kernel_expansion}, i.e.,
\begin{equation}\label{e:collocation-solution}
\hat{u}(\vx)=\sum_{k=1}^Nc_{k}P_2\Kstar(\vx,\vx_k)+\sum_{k=1}^Mc_{N+k}B_2\Kstar(\vx,\vx_{N+k}),\quad \vx\in\Domain,
\end{equation}
where $\vc:=(c_1,\cdots,c_{N+M})^T=\vKstar_{PB}{}^{-1}\vy_0\in\RR^{N+M}$.

Finally, we can perform a weak error analysis for $\abs{u(\vx)-\hat{u}(\vx)}$ as in Proposition~\ref{prop:main-approximation}, and deduce that
\[
\PP^{\mu}\left(\Eset_{\vx}^{\epsilon}\right)=
\PP^{\mu}\left(\abs{S_{\vx}-\hat{u}(\vx)}\geq\epsilon|\vS_{PB}=\vy_0\right)
=\erfc\left(\frac{\epsilon}{\sqrt{2}\sigma(\vx)}\right),
\]
where $\sigma(\vx)^2$ is defined in Corollary~\ref{c:cond-prob-density}, and
$$
\Eset_{\vx}^{\epsilon}:=\left\{\omega\in\Omega_K:\abs{\omega(\vx)-\hat{u}(\vx)}\geq\epsilon\text{ s.t. }
P\omega(\vx_1)=y_1,\ldots,B\omega(\vx_{N+M})=y_{N+M}\right\}.
$$
%with $\vomega_{PB}:=(P\omega(\vx_1),\cdots,P\omega(\vx_{N}),B\omega(\vx_{N+1}),\cdots,B\omega(\vx_{N+M}))^T$.

Because the form of the expression for the variance $\sigma(\vx)^2$ is analogous to that of the \emph{power function},
we can use the same techniques as in the proofs from \cite{Fasshauer2007,Wendland2005}  to obtain a formula for the order of $\sigma(\vx)$.
%////////////////////////////////////////////////////////////////////////////////////////////////////////////////////////
\begin{lemma}\label{l:powerfun-order}
When $P$ is the second-order elliptic differential operator and $B$ is the Dirichlet boundary condition, then
\[
\sigma(\vx)=\Order(h_X^{m-\rho-d/2}),\quad \vx\in\Domain,
\]
where $\rho:=\max\{\Order(P),\Order(B)\}$ and $h_X$ is the fill distance of $X_{\Domain}$ and $X_{\partial\Domain}$.
\end{lemma}
%////////////////////////////////////////////////////////////////////////////////////////////////////////////////////////
\begin{proof}
Since there is at least one collocation point $\vx_j\in X_{\Domain}\cup X_{\partial\Domain}$ such that $\norm{\vx-\vx_j}_2\leq h_X$ we can use the multivariate Taylor expansion of $\Kstar(\vx,\vx_j)$ to introduce the order of $\sigma(\vx)$, i.e.,
\[
\Kstar(\vx,\vx_j)=\sum_{\abs{\alpha},\abs{\beta}<n}\frac{1}{\alpha!\beta!}D_1^{\alpha}D_2^{\beta}\Kstar(\vx_j,\vx_j)(\vx-\vx_j)^{\alpha+\beta}+R(\vx,\vx_j),
\quad \alpha,\beta\in\NN_0^d,
\]
where $R(\vx,\vx_j):=\sum_{\abs{\alpha},\abs{\beta}=n}\frac{1}{\alpha!\beta!}D_1^{\alpha}D_2^{\beta}\Kstar(\vz_{1},\vz_{2})(\vx-\vx_j)^{\alpha+\beta}$ for some $\vz_{1},\vz_{2}\in\Domain$ and $n:=\lceil m -d/2\rceil-1$. The rest of the proof proceeds as in \cite[Chapter~14.5]{Fasshauer2007} and \cite[Chapters~11.3, 16.3]{Wendland2005}.
\end{proof}
%////////////////////////////////////////////////////////////////////////////////////////////////////////////////////////

Using Lemma~\ref{l:powerfun-order} we can deduce the following proposition.
%////////////////////////////////////////////////////////////////////////////////////////////////////////////////////////
\begin{proposition}
When $P$ is the second-order elliptic differential operator and $B$ is the Dirichlet boundary condition, then, 
for any $\epsilon>0$,
\[
\sup_{\mu\in\Hilbert_K(\Domain)}\PP^{\mu}(\Eset_{\vx}^{\epsilon})=\Order\left(\frac{h_X^{m-\rho-d/2}}{\epsilon}\right),
\quad \vx\in\Domain,
\]
which indicates that
\[
\sup_{\mu\in\Hilbert_K(\Domain)}\PP^{\mu}\left(\norm{u-\hat{u}}_{\Leb_{\infty}(\Domain)}\geq\epsilon\right)\leq
\sup_{\mu\in\Hilbert_K(\Domain),\vx\in\Domain}\PP^{\mu}\left(\Eset_{\vx}^{\epsilon}\right)\rightarrow 0,
 \text{ when }h_X\rightarrow0.
\]
\end{proposition}
%////////////////////////////////////////////////////////////////////////////////////////////////////////////////////////
Therefore we say that the estimator $\hat{u}$ converges to the exact solution $u$ of the PDE~(\ref{e:PDE-ell-no-noise}) in all probabilities $\PP^{\mu}$ when $h_X$ goes to $0$.

Sometimes we know only that the solution $u\in\Hil^{m}(\Domain)$. In this case, as long as the reproducing kernel Hilbert space is dense in the Sobolev space $\Hil^{m}(\Domain)$ with respect to its Sobolev norm, we can still say that $\hat{u}$ converges to $u$ in probability.

\subsection{Elliptic stochastic PDEs}\label{s:Int-SPDE}

Let $\xi:\Domain\times\Omega_W\rightarrow\RR$ be Gaussian with mean $0$ and covariance kernel $\Psi:\Domain\times\Domain\rightarrow\RR$ on the probability space $(\Omega_W,\Filter_W,\PP_W)$.
We consider an  elliptic PDE driven by a Gaussian additive noise $\xi$
\begin{equation}\label{e:PDE-ell-noise}
\begin{cases}
Pu=f+\xi,&\text{in }\Domain\subset\Rd,\\
Bu=g,&\text{on }\partial\Domain,
\end{cases}
\end{equation}
and suppose its solution $u\in\Leb_2(\Omega_W;\Hilbert_K(\Domain))$.

Since $\xi$ is a Gaussian process, on some underlying probability space $(\Omega_W,\Filter_W,\PP_W)$ with a known correlation structure, we can simulate the values of $\xi$ at $\vx_j, \ j=1,\ldots,N$. Consequently, we assume that the values $\{y_j\}_{j=1}^N$ and $\{y_{N+k}\}_{k=1}^M$ defined by
\[
y_j:=f(\vx_j)+\xi_{\vx_j},\quad y_{N+k}:=g(\vx_{N+k}), \quad j=1,\cdots,N,\ k=1,\cdots,M,
\]
are known. 
In this case, $(y_1,\cdots,y_N)\sim\Normal(\vf,\vPsi)$, where $\vf:=(f(\vx_1),\cdots,f(\vx_N))^T$ and $\vPsi:=(\Psi(\vx_j,\vx_k))_{j,k=1}^{N,N}$ with $\Psi$ being the covariance  kernel of $\xi$.
Let
\[
\vy_{\xi}:=\left(y_1,\cdots,y_{N+M}\right)^T,
\]
and $p_{\vy}$ be the probability density function of the random vector $\vy_{\xi}$.

In order to apply the general interpolation framework developed in Section~\ref{sec:StochasticProblems}, we consider the product space
\[
\Omega_{KW}:=\Omega_K\times\Omega_{W},\quad
\Filter_{KW}:=\Filter_K\otimes\Filter_{W},\quad
\PP_{W}^{\mu}:=\PP^{\mu}\otimes\PP_{W}.
\]
We assume that the random variables defined on the original probability spaces are extended to random variables on the new probability space in the natural way: if random variables $V_1:\Omega_K\rightarrow\RR$ and
$V_2:\Omega_W\rightarrow\RR$ are defined on $(\Omega_K,\Filter_K,\PP^{\mu})$ and $(\Omega_W,\Filter_W,\PP_W)$, respectively, then
\[
V_1(\omega,\tilde{\omega}):=V_1(\omega),\quad
V_2(\omega, \tilde{\omega}):=V_2(\tilde{\omega}),\quad
\text{for each }\omega\in\Omega_K\text{ and }\tilde{\omega}\in\Omega_W.
\]
Note that in this case the random variables have the same probability distributional properties, and they are independent on $(\Omega_{KW},\Filter_{KW},\PP_W^{\mu})$. This implies that the stochastic processes $S$, $PS$, $BS$ and $\xi$ can be extended to the product space $(\Omega_{KW},\Filter_{KW},\PP_W^{\mu})$ while preserving the original probability distributional properties.

Fix any $\vx\in\Domain$. Let $\Aset_{\vx}(v):=\left\{\omega_1\times\omega_2\in\Omega_{KW}:\omega_1(\vx)=v\right\}$ for each $v\in\RR$, and
$\Aset_{PB}^{\vy_{\xi}}:=\left\{\omega_1\times\omega_2\in\Omega_{KW}:P\omega_1(\vx_1)=y_1(\omega_2),\ldots,B\omega_1(\vx_{N+M})=y_{N+M}(\omega_2)\right\}$.
Using the methods in Section~\ref{s:Int-PDE} and Theorem~\ref{t:Gauss-RK-PB}, we obtain
\[
\PP_{W}^{\mu}(\Aset_{\vx}(v)|\Aset_{PB}^{\vy_{\xi}})=\PP_{W}^{\mu}(S_{\vx}=v|\vS_{PB}=\vy_{\xi})=p_{\vx}^{\mu}(v|\vy_{\xi}),
\]
where $p^{\mu}_{\vx}(\cdot|\cdot)$ is the conditional probability density function of the random variable $S_{\vx}$ given the random vector $\vS_{PB}:=\left(P S_{\vx_1},\cdots,P S_{\vx_N},B S_{\vx_{N+1}},\cdots,B S_{\vx_{N+M}}\right)^T$. (Here $\vy_{\xi}$ is viewed as given values.) According to the natural extension rule, $p_{\vx}^{\mu}$ is consistent with the formula in Corollary~\ref{c:cond-prob-density}.
If $\vKstar_{PB}$ is nonsingular, then the approximation $\hat{u}(\vx)$ is solved by the maximization problem
\[
\hat{u}(\vx)=\argmax_{v\in\RR}\sup_{\mu\in\Hilbert_K(\Domain)}p_{\vx}^{\mu}(v|\vy_{\xi})
=\vk_{PB}(\vx)^T\vKstar_{PB}{}^{-1}\vy_{\xi},
\]
where $\vKstar_{PB}$ and $\vk_{PB}(\vx)$ are defined in Corollary~\ref{c:cov-matrix-Gauss-RK-PB} and~\ref{c:cond-prob-density}.
This means that its random coefficients are obtained from the linear equation system 
\[
\vKstar_{PB}\vc=\vy_{\xi}. 
\]

The estimator $\hat{u}$ also satisfies the interpolation condition, i.e., $P\hat{u}(\vx_1)=y_1,\ldots,P\hat{u}(\vx_N)=y_N$ and $B\hat{u}(\vx_{N+1})=y_{N+1},\ldots,B\hat{u}(\vx_{N+M})=y_{N+M}$.
It is obvious that $\hat{u}(\cdot,\omega_2)\in\Hilbert_K(\Domain)$ for each $\omega_2\in\Omega_{W}$.
Since the random part of $\hat{u}(\vx)$ is only related to $\vy_{\xi}$,
we can formally rewrite $\hat{u}(\vx,\omega_2)$ as $\hat{u}(\vx,\vy_{\xi})$ and $\hat{u}(\vx)$ can be transferred to a random variable defined on the finite-dimensional probability space $(\RR^{N+M},\Borel(\RR^{N+M}),\mu_{\vy})$, where the probability measure $\mu_{\vy}$ is defined by $\mu_{\vy}(\ud\vv):=p_{\vy}(\vv)\ud\vv$. Moreover, the probability distributional properties of $\hat{u}(\vx)$ do not change when $(\Omega_{W},\Filter_{W},\PP_{W})$ is replaced by $(\RR^{N+M},\Borel(\RR^{N+M}),\mu_{\vy})$.

Finally, we discuss the convergence analysis of this estimator. 
We assume that $u(\cdot,\omega_2)$ belongs to $\Hilbert_K(\Domain)$ almost surely for $\omega_2\in\Omega_W$. Therefore $u$ can be seen as a map from $\Omega_{W}$ into $\Hilbert_K(\Domain)$. So we have $u\in\Omega_{KW}=\Omega_K\times\Omega_{W}$.

We fix any $\vx\in\Domain$ and any $\epsilon>0$. Let the subset
\[
\begin{split}
\Eset_{\vx}^{\epsilon}:=&\Big\{\omega_1\times\omega_2\in\Omega_{KW}:\abs{\omega_1(\vx)-\hat{u}(\vx,\omega_2)}\geq\epsilon, \\
&\text{ such that }
P\omega_1(\vx_1)=y_1(\omega_2),\ldots,B\omega_1(\vx_{N+M})=y_{N+M}(\omega_2)\Big\}.
\end{split}
\]
Because $P S_{\vx}(\omega_1,\omega_2)=P S_{\vx}(\omega_1)=P\omega_1(\vx)$, $B S_{\vx}(\omega_1,\omega_2)=B S_{\vx}(\omega_1)=B\omega_1(\vx)$ and $\vy_{\vxi}(\omega_1,\omega_2)=\vy_{\vxi}(\omega_2)$ for each $\omega_1\in\Omega_K$ and $\omega_2\in\Omega_{W}$ (see~Theorem~\ref{t:Gauss-RK-PB})
we can deduce that
\[
\begin{split}
\PP^{\mu}_{W}\left(\Eset_{\vx}^{\epsilon}\right)&=
\PP^{\mu}_{W}\left(\abs{S_{\vx}-\hat{u}(\vx)}\geq\epsilon \text{ such that } \vS_{PB}=\vy_{\xi}\right)\\
&=\int_{\RR^{N+M}}\int_{\abs{v-\hat{u}(\vx,\vv)}\geq\epsilon}p_{\vx}^{\mu}(v|\vv)p_{\vy}(\vv)\ud v\ud\vv\\
&=\int_{\RR^{N+M}}\erfc\left(\frac{\epsilon}{\sqrt{2}\sigma(\vx)}\right)p_{\vy}(\vv)\ud\vv
=\erfc\left(\frac{\epsilon}{\sqrt{2}\sigma(\vx)}\right),
\end{split}
\]
where the variance of $p^{\mu}_{\vx}$ is $\sigma(\vx)^2 =\Kstar(\vx,\vx)-\vk_{PB}(\vx)^T \vKstar_{PB}{}^{-1}\vk_{PB}(\vx)$ (see Corollary~\ref{c:cond-prob-density}).

Similar to the analysis of the error bounds from Section~\ref{s:Int-PDE}, we also deduce the following proposition (for more details see \cite{YePhD2012}).
%////////////////////////////////////////////////////////////////////////////////////////////////////////////////////////
\begin{proposition}
When $P$ is the second-order elliptic differential operator and $B$ is the Dirichlet boundary condition, then,
\[
\lim_{h_X\rightarrow0}\sup_{\mu\in\Hilbert_K(\Domain)}\PP^{\mu}_W\left(\norm{u-\hat{u}}_{\Leb_{\infty}(\Domain)}\geq\epsilon \right)=0,\quad
\text{for any }\epsilon>0.
\]
\end{proposition}
%////////////////////////////////////////////////////////////////////////////////////////////////////////////////////////

\section{Numerical Experiments}\label{sec:numerical-experimants}

We consider the following stochastic heat equation with zero boundary condition
\begin{equation}\label{eq:num-exa}
\begin{cases}
\ud U_t=\frac{\ud^2}{\ud x^2} U_t \ud t+ \sigma\ud W_{t,i},\quad\text{in }\Domain:=(0,1)\subset\RR,\quad 0<t<T:=1,\\
~~U_t=0,~~~~~~~~~~~~~~~~~~~~~~~~\text{on }\partial\Domain,\\
U_0(x)=u_0(x):=\sqrt{2}\left(\sin(\pi x)+\sin(2\pi x)+\sin(3\pi x)\right),
\end{cases}
\end{equation}
driven by two types of space-time white noise (colored in space) $W$ of the form
\[
W_{t,i}:=\sum_{k=1}^{\infty}W^k_tq_k^i\phi_k,\quad
q_k:=\frac{1}{k\pi},\quad \phi_k(x):=\sqrt{2}\sin(k\pi x),
\]
where $W_t^k, \ k\in\mathbb{N}$, is a sequence of independent one-dimensional Brownian motions, and $i=1,2$. Note that choosing the larger value of $i$ corresponds to a noise that is smoother in space.

The spatial covariance function $R^i(x,y)=\sum_{k=1}^{\infty}q_k^{2i}\phi_k(x)\phi_k(y)$, $i=1,2$, takes the specific forms
\[
R^1(x,y)=\min\{x,y\}-xy,\quad 0<x,y<1,
\]
and
\[
R^2(x,y)=
\begin{cases}
-\frac{1}{6}x^3+\frac{1}{6}x^3y+\frac{1}{6}xy^3-\frac{1}{2}xy^2+\frac{1}{3}xy,&0<x<y<1,\\
-\frac{1}{6}y^3+\frac{1}{6}xy^3+\frac{1}{6}x^3y-\frac{1}{2}x^2y+\frac{1}{3}xy,&0<y<x<1.
\end{cases}
\]

The solution of SPDE \eqref{eq:num-exa} is given by (for more details see, for instance, \cite{Chow2007})
\[
U_t(x)=\sum_{k=1}^{\infty}\xi^k_t\phi_k(x),\quad x\in\Domain:=(0,1),\quad 0<t<T:=1,
\]
where
\[
\xi^k_0:=\int_{\Domain}u_0(x)\phi_k(x)\ud x,\quad
\xi^k_t:=\xi^k_0\me^{-k^2\pi^2 t}+\frac{\sigma}{q_k^i}\int_0^t\me^{k^2\pi^2(s-t)}\ud W_s^k.
\]
From this explicit solution we can get that
\[
\mathbb{E}(U_t(x))=\sum_{k=1}^{\infty}\xi^k_0\me^{-k^2\pi^2 t}\phi_k(x),\quad
\Var(U_t(x))=\sum_{k=1}^{\infty}\frac{\sigma^2}{2k^2\pi^2 q_k^{2i} }(1-\me^{-2k^2\pi^2 t})\abs{\phi_k(x)}^2.
\]

%------------------------------------------------------------------------------------------------------------------------
%------------------------------------------------------------------------------------------------------------------------

We discretize the time interval $[0,T]$ with $n$ equal time steps so that $\delta t:=T/n$. We also choose the reproducing kernel $K(x,y):=g_{3,2\theta}(x-y)$, where $g_{3,2\theta}$ is the Mat\'ern function with degree $m:=3$ and shape parameter $\theta>0$ (see Example~\ref{ex:Sobolev-spline}). As collocation points we select uniform grid points $X_{\Domain}\subset(0,1)$ and $X_{\partial\Domain}:=\{0,1\}$. Let $P:=I-\delta t\ud^2/\ud x^2$ and $B:=I|_{\{0,1\}}$.
Using our kernel-based collocation method we can perform the following computations to numerically estimate the sample paths $\hat{u}^n_j\approx U_{t_n}(\vx_j)$.
\emph{Algorithm to solve SPDE~(\ref{eq:num-exa}):}
\begin{enumerate}
\item Initialize
\begin{itemize}
\item $\displaystyle \hat{\vu}^0:=\left(u_{0}(x_1),\cdots,u_{0}(x_N)\right)^T$
\item $\displaystyle
\vKstar_{PB}:=
\begin{pmatrix}
(P_1P_2\Kstar(x_j,x_k))_{j,k=1}^{N,N}, & (P_1B_2\Kstar(x_j,x_{N+k}))_{j,k=1}^{N,M}\\
(B_1P_2\Kstar(x_{N+j},x_k))_{j,k=1}^{M,N}, & (B_1B_2\Kstar(x_{N+j},x_{N+k}))_{j,k=1}^{M,M}\\
\end{pmatrix}
$
\item $\displaystyle
\vB:=\begin{pmatrix}
(P_2\Kstar(x_j,x_k))_{j,k=1}^{N,N}, & (B_2\Kstar(x_j,x_{N+k}))_{j,k=1}^{N,M}\\
\end{pmatrix}
$
\item $\displaystyle
\vPsi:=\sigma^2\delta t(R(x_j,x_k))_{j,k=1}^{N,N}
$
\item $\displaystyle
\vA:=\vB\vKstar_{PB}{}^{-1}
$
\end{itemize}
\item Repeat for $j=1,2,\ldots,n$, i.e., for $t_1,t_2,\ldots,t_n=T$
\begin{itemize}
\item Simulate
$\displaystyle
\vnoise\sim\Normal\left(0,\vPsi\right)$
\item $\displaystyle
 \hat{\vu}^j:=\vB \vKstar_{PB}{}^{-1}\begin{pmatrix}\hat{\vu}^{j-1}+\vnoise\\0\end{pmatrix}=\vA\begin{pmatrix}\hat{\vu}^{j-1}+\vnoise\\0\end{pmatrix}
$\end{itemize}
\end{enumerate}

Note that in the very last step the matrix $\vA$ is pre-computed and can be used for all time steps, and for different sample paths; that makes the proposed algorithm to be quite efficient.

We approximate the mean and variance of $U_t(x)$ by sample mean and sample variance from $s:=10000$ simulated sample paths using the above algorithm, i.e.,
\[
\mathbb{E}(U_{t_n}(x_j))\approx \frac{1}{s}\sum_{k=1}^{s}\hat{u}_j^n(\omega_k),\quad
\Var(U_{t_n}(x_j))\approx \frac{1}{s}\sum_{i=1}^{s}\left(\hat{u}_j^n(\omega_i)-\frac{1}{s}\sum_{k=1}^{s}\hat{u}_j^n(\omega_k)\right)^2.
\]

\begin{figure}[h]
\begin{center}
  \centering
  \subfigure[With spatial covariance $R^1$]{
    \includegraphics[width=0.48\textwidth, height=0.45\textwidth]{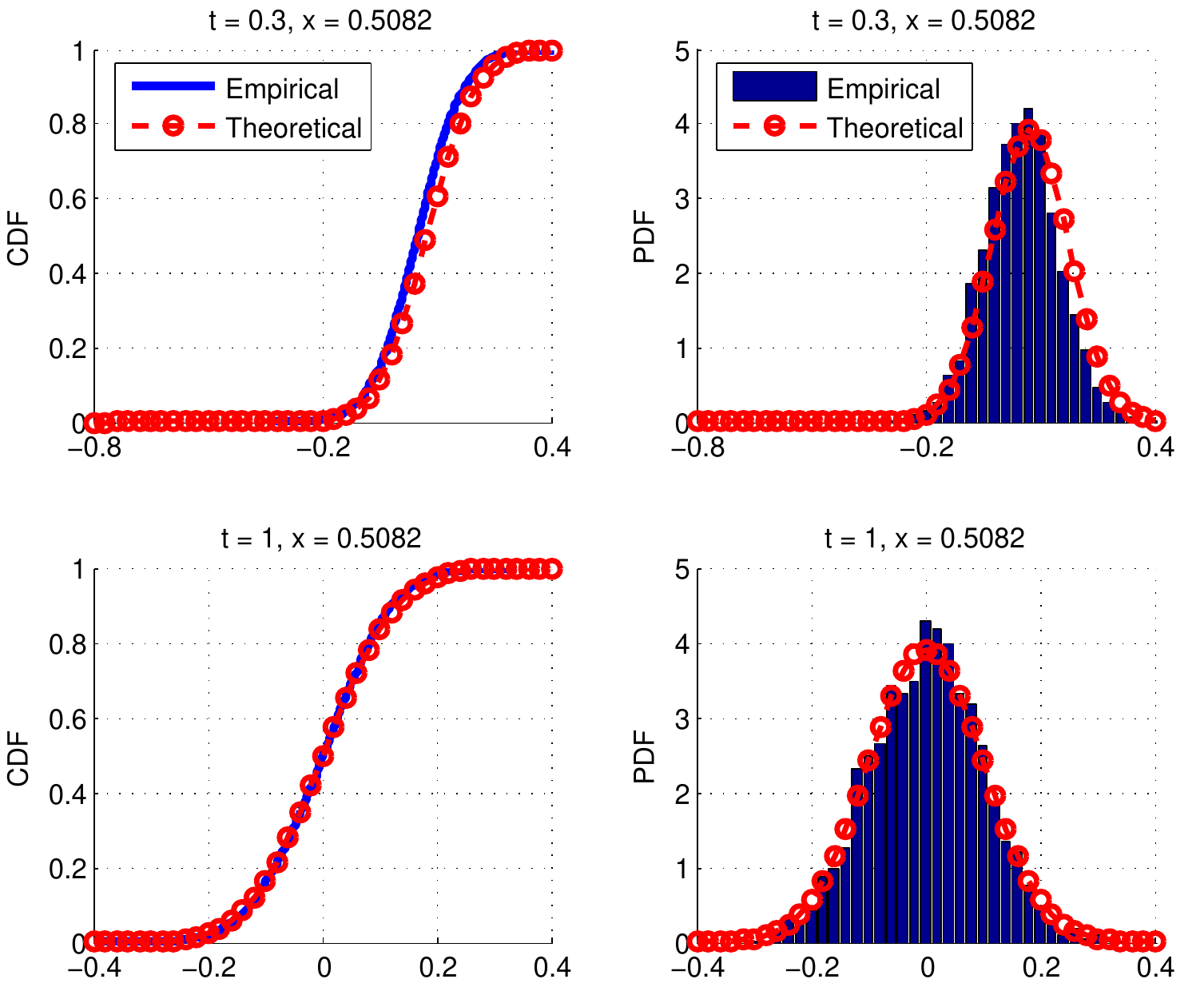}}
  \subfigure[With spatial covariance $R^2$]{
    \includegraphics[width=0.48\textwidth, height=0.45\textwidth]{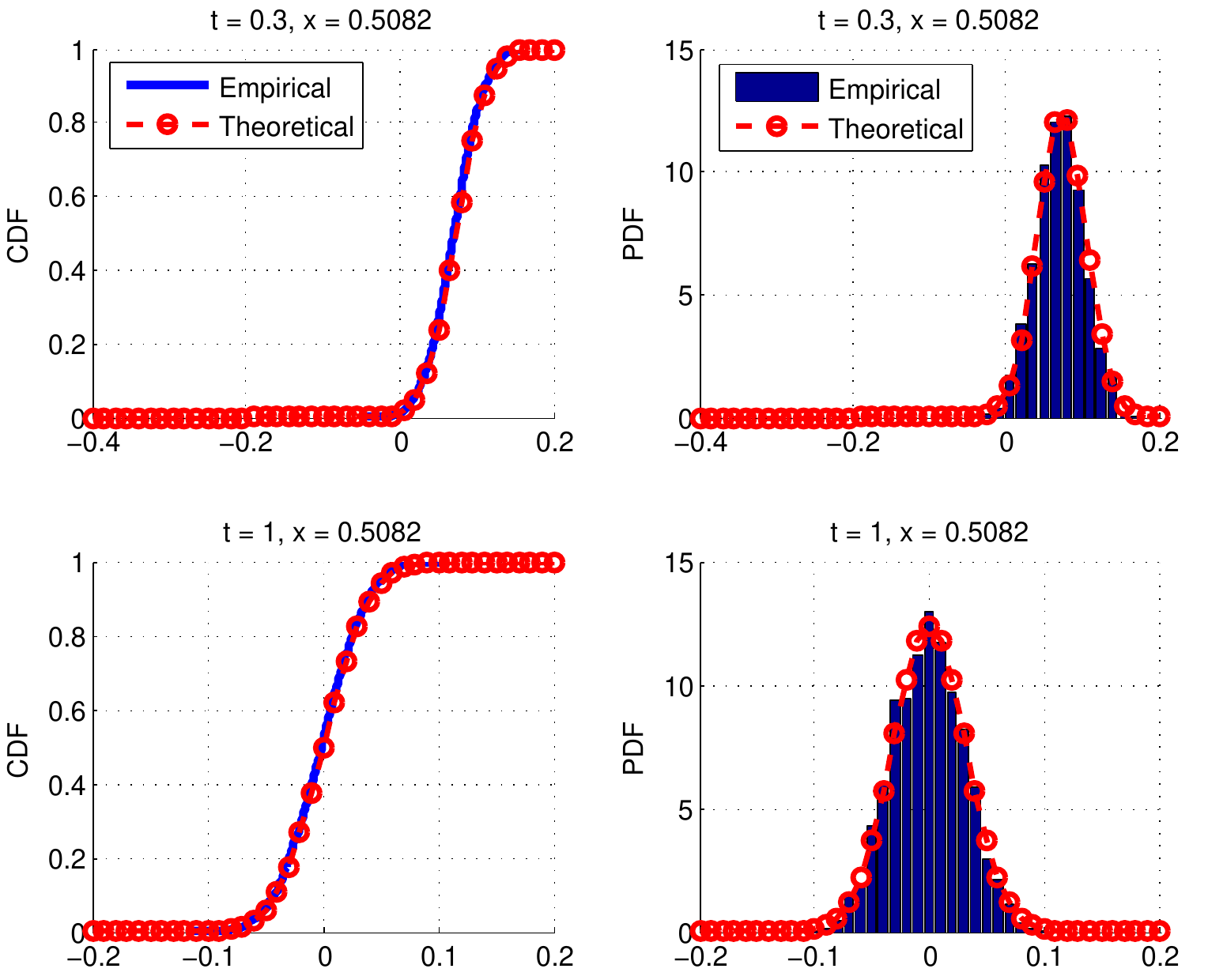}}
\caption{Empirical and theoretical probability distribution of $U_t(x)$ for uniform points $N:=58$ and $M:=2$, equal time steps $n:=800$, $\theta:=26.5$, $\sigma:=1$.}\label{fig:distribution}
\end{center}
\end{figure}

\begin{figure}[h]
\begin{center}
  \centering
  \subfigure[With spatial covariance $R^1$]{
    \includegraphics[width=0.48\textwidth, height=0.45\textwidth]{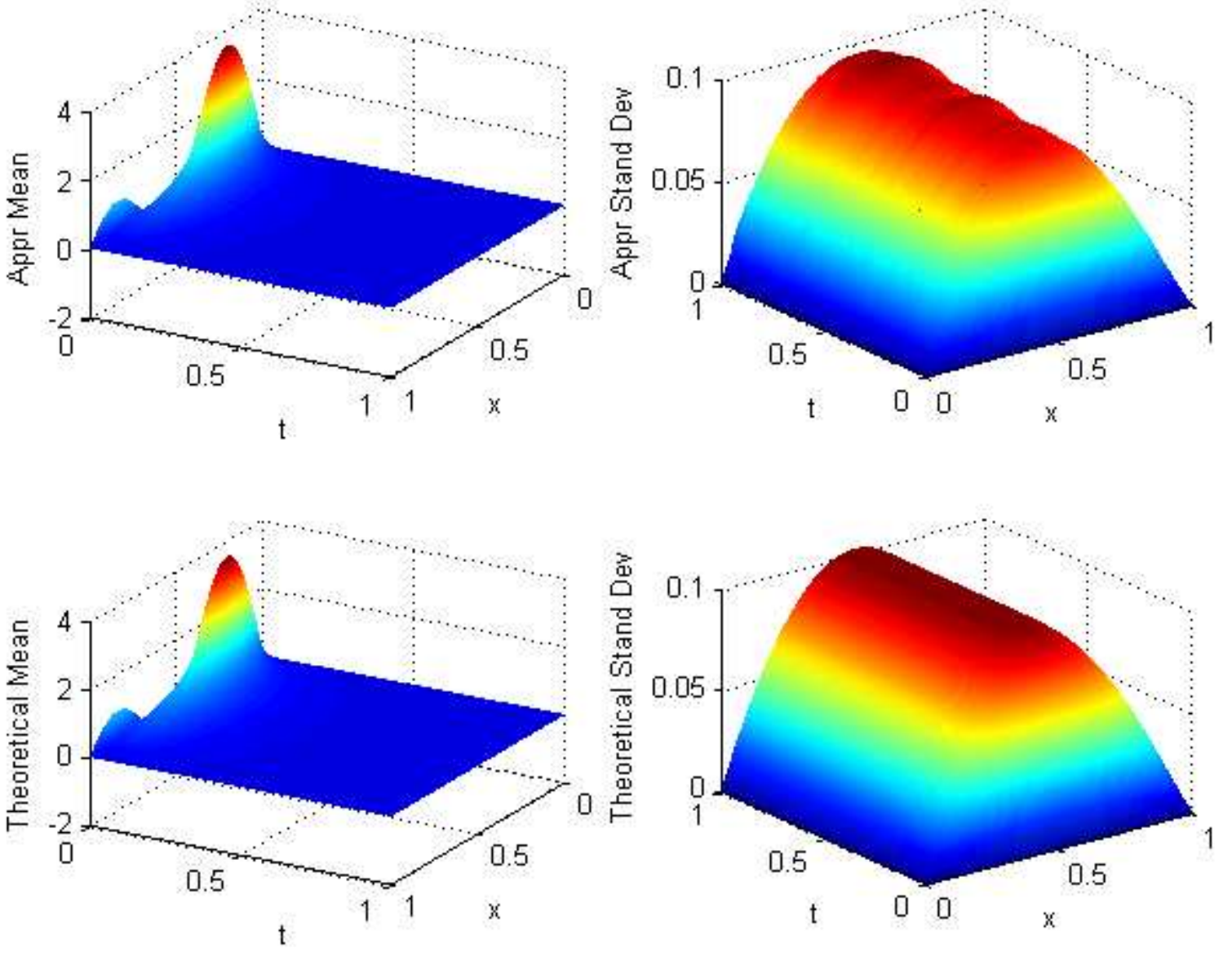}}
  \subfigure[With spatial covariance $R^2$]{
    \includegraphics[width=0.48\textwidth, height=0.45\textwidth]{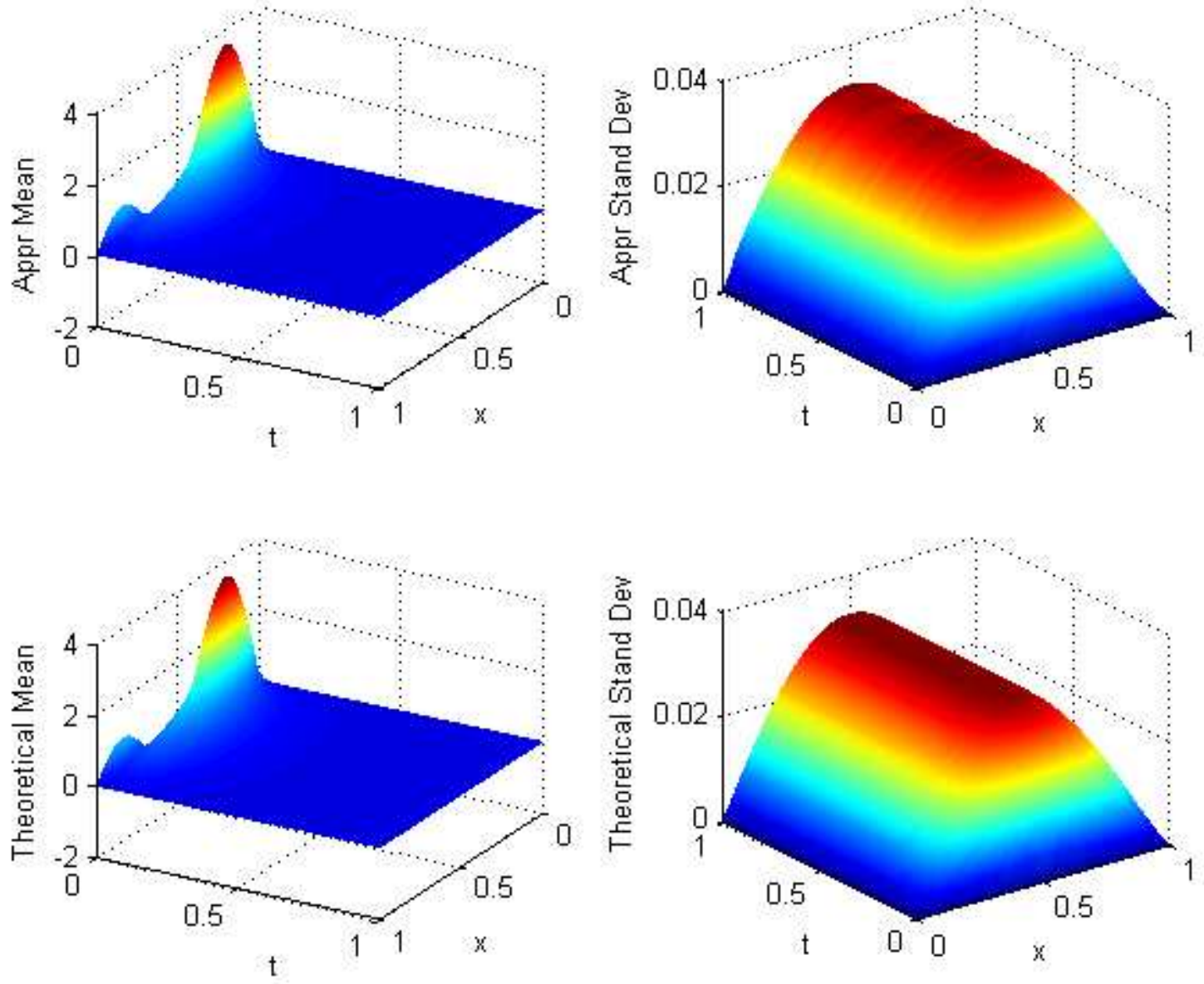}}
\caption{Approximate and theoretical mean and standard deviation for uniform points $N:=58$ and $M:=2$, equal time steps $n:=800$, $\theta:=26.5$, $\sigma:=1$.}\label{fig:mean-variance}
\end{center}
\end{figure}

Figure~\ref{fig:distribution} shows that the  histograms at different values of $t$ and $x$ resemble the theoretical normal distributions.
We notice a small shift in the probability distribution function of the solution $U$, at times closer to zero, and when the noise is equal to $W_1$ (Figure ~\ref{fig:distribution}, left panel). This shift is due to the fact that $W_1$ is rougher in space than $W_2$.

Our use of an implicit time stepping scheme reduces the frequency of the white noise, i.e., $\lim_{\delta t\rightarrow0}\delta W/\delta t\sim\delta_0$. Consequently, Figure~\ref{fig:mean-variance} shows that the approximate mean is well-behaved but the approximate variance is a little smaller than the exact variance.
According to Figure~\ref{fig:time-h-order} we find that this numerical method is convergent as both $\delta t$ and $h_X$ are refined. Finally, we want to mention that the distribution of collocation points, the shape parameter, and the kernel itself were chosen empirically and based on the authors' experience. As mentioned before, more precise methods are currently not available.  A rigorous investigation of these questions, as well as determination of precise rates of convergence is reserved for future work.

\begin{figure}
  \centering
  \subfigure[With spatial covariance $R^1$]{
    \includegraphics[width=0.48\textwidth, height=0.45\textwidth]{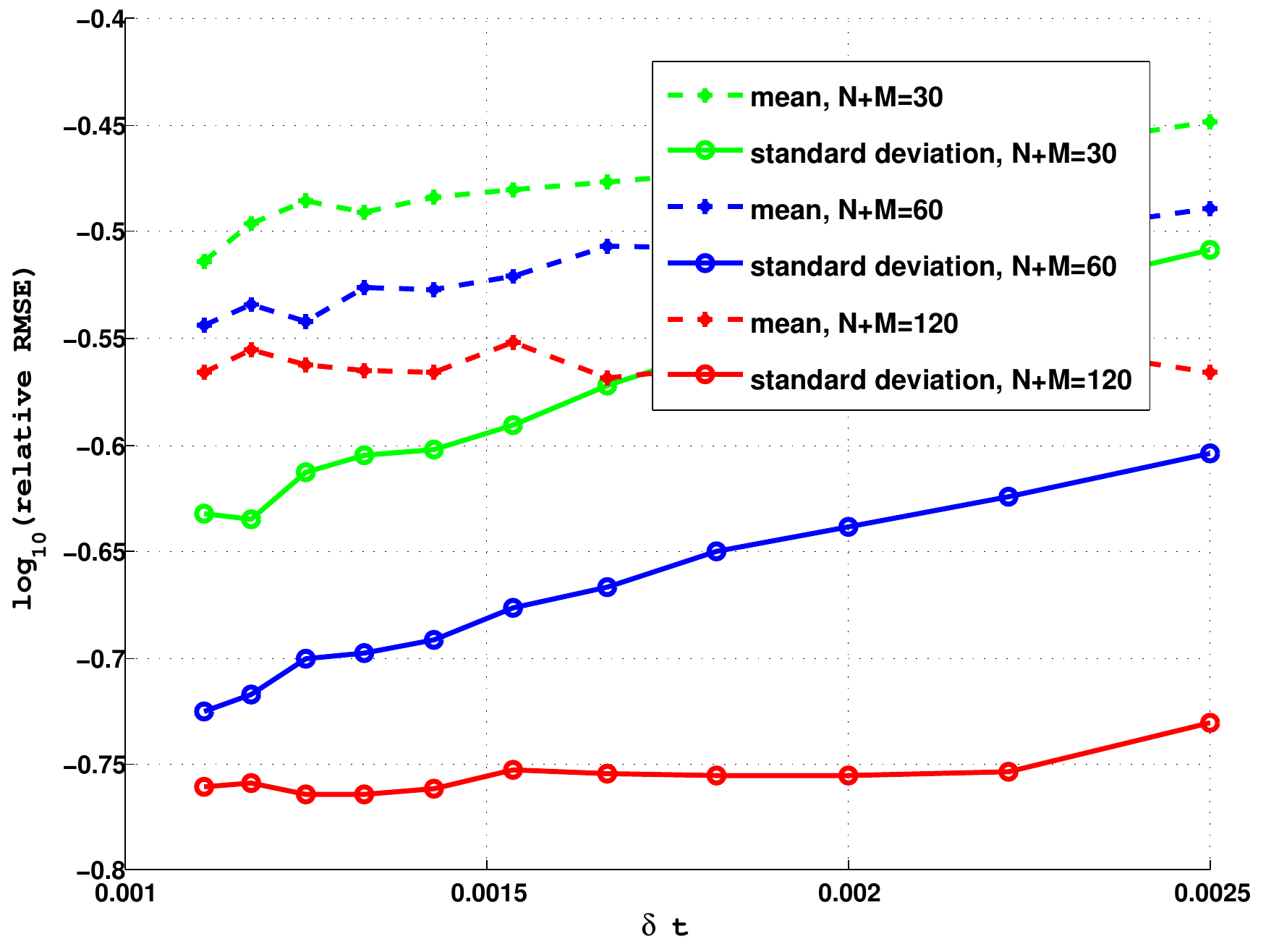}}
  \subfigure[With spatial covariance $R^2$]{
    \includegraphics[width=0.48\textwidth, height=0.45\textwidth]{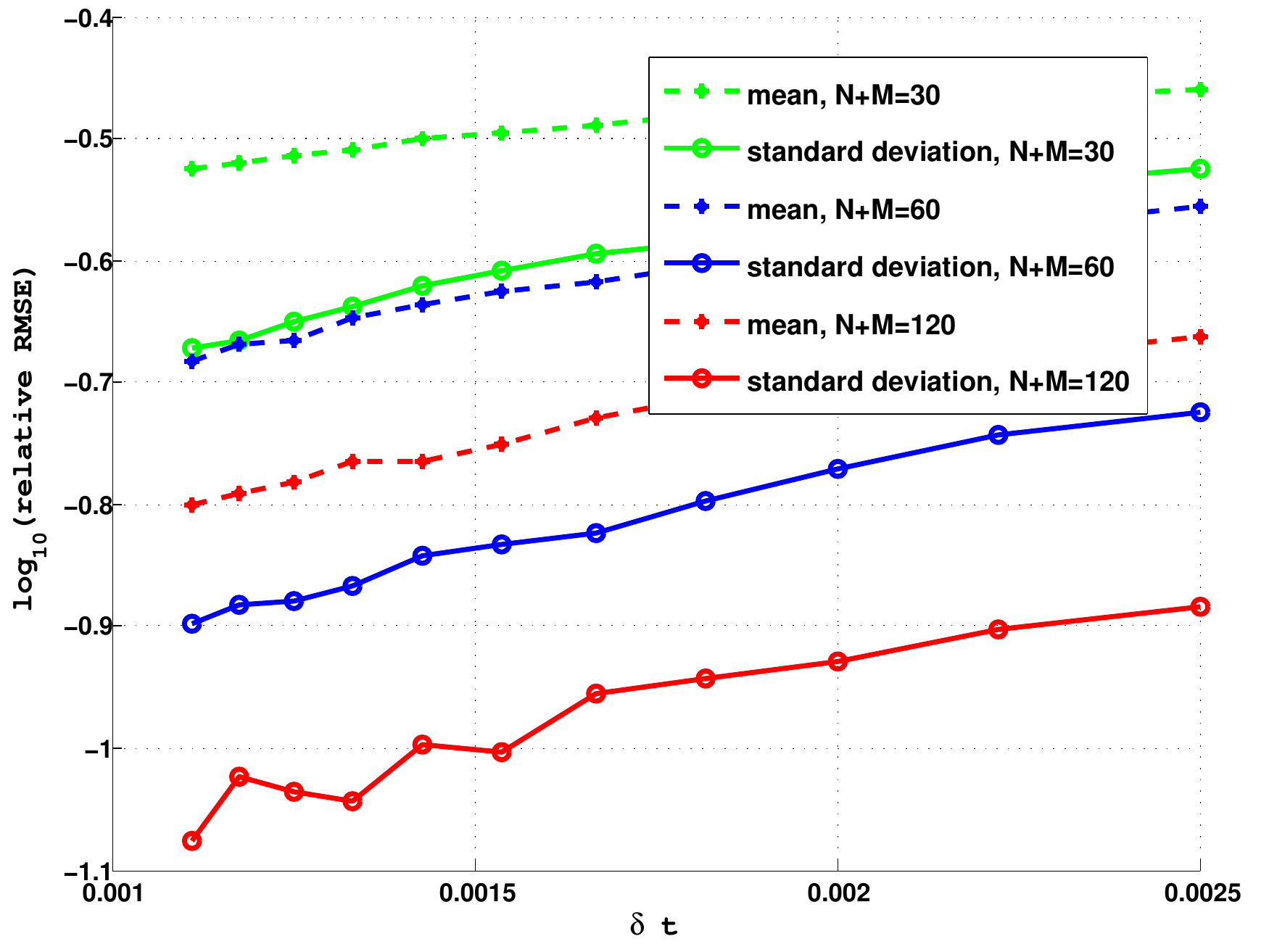}}
\caption{Convergence of mean and variance with respect to refinement of points and time steps for $\sigma:=1$.
(The relative RMSE of exact $U$ and approximate $\hat{U}$ is defined by $\mathrm{RMSE}(U,\hat{U}):=\sqrt{\frac{1}{nN}\sum_{j=1}^n\sum_{k=1}^N(U(t_j,x_k)-\hat{U}(t_j,x_k))^2/\norm{U(t_k,\cdot)}_{\infty}^2}$.)
}\label{fig:time-h-order}
\end{figure}

%---------------------------------------------------------------------------------------------------------------------
%/////////////////////////////////////////////////////////////////////////////////////////////////////////////////////
%---------------------------------------------------------------------------------------------------------------------

\section{Final Remarks}

This new numerical approach can also be used to approximate systems of elliptic PDEs with vector Gaussian noises $\vnoise_1$ and $\vnoise_2$ or nonlinear PDEs with Gaussian noise $\xi$, i.e.,
\[
\begin{cases}
\mathbf{P} u=\vf+\vnoise_1,&\text{in }\Domain\subset\Rd,\\
\mathbf{B} u=\vg+\vnoise_2,&\text{on }\partial\Domain,
\end{cases}
\quad\text{or}\quad
\begin{cases}
F(\mathbf{P} u)=\psi(f,\xi),&\text{in }\Domain\subset\Rd,\\
G(\mathbf{B} u)=g,&\text{on }\partial\Domain,
\end{cases}
\]
where $\mathbf{P}:=(P^1,\cdots,P^{\np})^T$ is a vector differential operator and $\mathbf{B}:=(B^1,\cdots,B^{\nb})^T$ is a vector boundary operator,
and $F:\RR^{\np}\rightarrow\RR$ and $G:\RR^{\nb}\rightarrow\RR$ (see~\cite{YePhD2012}).

In addition to the additive noise case discussed here, we can also use the kernel-based collocation method to approximate other well-posed stochastic parabolic equations with multiplicative noise, e.g.,
\begin{equation}\label{e:SPDE-par-extend}
\begin{cases}
\ud U_t=\mathcal{A}U_t \ud t+ \psi(U_t) \ud W_t,&\text{in }\Domain\subset\Rd,\quad 0<t<T,\\
BU_t=0,&\text{on }\partial\Domain,\\
U_0=u_0,
\end{cases}
\end{equation}
where $\psi:\RR\rightarrow\RR$. Since $\int_{t_{j-1}}^{t_j}\psi(U_{s})\ud W_s\approx\psi(U_{t_{j-1}})\delta W_j$, the algorithm for SPDE~(\ref{e:SPDE-par-extend}) is similar to before:
\begin{enumerate}
\item Initialize
\begin{itemize}
\item $\displaystyle \hat{\vu}^0:=\left(u_{0}(\vx_1),\cdots,u_{0}(\vx_N)\right)^T$
\item $\displaystyle
\vKstar_{PB}:=
\begin{pmatrix}
(P_1P_2\Kstar(\vx_j,\vx_k))_{j,k=1}^{N,N}, & (P_1B_2\Kstar(\vx_j,\vx_{N+k}))_{j,k=1}^{N,M}\\
(B_1P_2\Kstar(\vx_{N+j},\vx_k))_{j,k=1}^{M,N}, & (B_1B_2\Kstar(\vx_{N+j},\vx_{N+k}))_{j,k=1}^{M,M}\\
\end{pmatrix}
$
\item $\displaystyle
\vB:=\begin{pmatrix}
(P_2\Kstar(\vx_j,\vx_k))_{j,k=1}^{N,N}, & (B_2\Kstar(\vx_j,\vx_{N+k}))_{j,k=1}^{N,M}\\
\end{pmatrix}
$
\item $\displaystyle
\vPsi_0:=\delta t(R(\vx_j,\vx_k))_{j,k=1}^{N,N}
$
\end{itemize}
\item Repeat for $j=1,2,\cdots,n$, i.e., for $t_1,~t_2,~\cdots,~t_n=T$
\begin{itemize}
\item $\displaystyle
\vV_1:=\diag\left(\psi(\hat{u}^{j-1}_1),\cdots,\psi(\hat{u}^{j-1}_N)\right)
$
\item $\displaystyle
\vPsi:=\vV_1 \vPsi_0 \vV_1
$
\item Simulate
$\displaystyle
\vnoise\sim\Normal\left(0,\vPsi\right),
\quad \vA:=\vB\vKstar_{PB}{}^{-1}$
\item $\displaystyle
 \hat{\vu}^j:=\vB\vKstar_{PB}{}^{-1}
 \begin{pmatrix}\hat{\vu}^{j-1}+\vnoise\\0\end{pmatrix}=\vA\begin{pmatrix}\hat{\vu}^{j-1}+\vnoise\\0\end{pmatrix}
$
\end{itemize}
\end{enumerate}

Of course, now the matrix $\vA$ needs to be updated for each time step and for each sample path so that the algorithm is much costlier.

\appendix

%---------------------------------------------------------------------------------------------------------------------
%/////////////////////////////////////////////////////////////////////////////////////////////////////////////////////
%---------------------------------------------------------------------------------------------------------------------

\section{Reproducing-Kernel Hilbert Spaces}\label{s:RKHS}

%////////////////////////////////////////////////////////////////////////////////////////////////////////////////////////

\begin{definition}[{\cite[Definition~10.1]{Wendland2005}}]\label{d:RKHS}
A Hilbert
space $\Hilbert_K(\Domain)$ of functions $f:\Domain\rightarrow\RR$ is
called a \emph{reproducing-kernel Hilbert space} with a
\emph{reproducing kernel} $K:\Domain\times\Domain\rightarrow\RR$ if
\[
(i)~K(\cdot,\vy)\in\Hilbert_K(\Domain)\quad \text{ and }\quad
(ii)~f(\vy)=\langle K(\cdot,\vy),f \rangle_{K,\Domain},
\]
for all $f\in\Hilbert_K(\Domain)$ and all $\vy\in\Domain$. Here $\langle\cdot,\cdot\rangle_{K,\Domain}$
is the inner product of $\Hilbert_K(\Domain)$.
\end{definition}
%////////////////////////////////////////////////////////////////////////////////////////////////////////////////////////

According to \cite[Theorem~10.4]{Wendland2005} all reproducing kernels are positive semi-definite. \cite[Theorem~10.10]{Wendland2005} shows that a symmetric positive semi-definite kernel $K$ is always a reproducing kernel of a reproducing-kernel Hilbert space $\Hilbert_K(\Domain)$.

If $\Domain$ is open and bounded (pre-compact) and $K\in\Leb_2(\Domain\times\Domain)$ is symmetric positive definite, then Mercer's
theorem~\cite[Theorem~13.5]{Fasshauer2007} guarantees the existence of a countable set of positive values $\lambda_1\geq\lambda_2\geq\cdots>0$ with $\sum_{k=1}^{\infty}\lambda_k<\infty$
and an orthonormal base $\{e_k\}_{k=1}^{\infty}$ of $\Leb_2(\Domain)$ such that
$K$ possesses the absolutely and uniformly convergent representation
\[
K(\vx,\vy)=\sum_{k=1}^{\infty}\lambda_ke_k(\vx)e_k(\vy),\quad \vx,\vy\in\Domain.
\]
This Mercer series of $K$ implies that
\[
\lambda_ke_k(\vy)=\int_{\Domain}K(\vx,\vy)e_k(\vx)\ud\vx,\quad \vy\in\Domain,\quad k\in\NN.
\]

%////////////////////////////////////////////////////////////////////////////////////////////////////////////////////////
\begin{definition}\label{d:eigval_fun}
The sequences $\{\lambda_k\}_{k=1}^{\infty}$ and $\{e_k\}_{k=1}^{\infty}$ given above are called the \emph{eigenvalues} and \emph{eigenfunctions} of the reproducing kernel $K$.
\end{definition}
%////////////////////////////////////////////////////////////////////////////////////////////////////////////////////////

Since $\{e_k\}_{k=1}^{\infty}$ is orthonormal in $\Leb_2(\Domain)$ we can compute the series expansion of the integral-type kernel $\Kstar$ defined in Lemma~\ref{l:Gauss-RK}, i.e.,
\[
\Kstar(\vx,\vy)
=\sum_{j=1}^{\infty}\sum_{k=1}^{\infty}\int_{\Domain}\lambda_je_j(\vx)e_j(\vz)\lambda_ke_k(\vz)e_k(\vy)\ud\vz
=\sum_{k=1}^{\infty}\lambda_k^2e_k(\vx)e_k(\vy).
\]
It is easy to check that $\Kstar\in\Leb_2(\Domain\times\Domain)$ is symmetric positive definite and $\{\lambda_k^2\}_{k=1}^{\infty}$ and $\{e_k\}_{k=1}^{\infty}$ are the eigenvalues and eigenfunctions of $\Kstar$.

%////////////////////////////////////////////////////////////////////////////////////////////////////////////////////////
\begin{theorem}[{\cite[Theorem~10.29]{Wendland2005}}]\label{t:RKHS}
Suppose that $K\in\Leb_2(\Domain\times\Domain)$ is a symmetric positive definite kernel on a pre-compact $\Domain\subset\Rd$.
Then its reproducing-kernel Hilbert space is given by
\[
\Hilbert_K(\Domain)=\left\{f\in\Leb_2(\Domain):\sum_{k=1}^{\infty}\frac{1}{\lambda_k}\abs{\int_{\Domain}f(\vx)e_k(\vx)\ud\vx}^2<\infty\right\}
\]
and the inner product has the representation
\[
\langle f,g \rangle_{K,\Domain}=\sum_{k=1}^{\infty}\frac{1}{\lambda_k}\int_{\Domain}f(\vx)e_k(\vx)\ud\vx\int_{\Domain}g(\vx)e_k(\vx)\ud\vx,\quad f,g\in\Hilbert_K(\Domain),
\]
where $\{\lambda_k\}_{k=1}^{\infty}$ and $\{e_k\}_{k=1}^{\infty}$ are the eigenvalues and eigenfunctions of $K$.
\end{theorem}
%////////////////////////////////////////////////////////////////////////////////////////////////////////////////////////
We can verify that $\{\sqrt{\lambda_k}e_k\}_{k=1}^{\infty}$ is an orthonormal base of $\Hilbert_K(\Domain)$.

%////////////////////////////////////////////////////////////////////////////////////////////////////////////////////////
\begin{example}\label{ex:Sobolev-spline}

The papers~\cite{FasshauerYe2011online,YeIITTech2010} show that the \emph{Sobolev spline} (Mat\'ern function) of degree $m>\frac{d}{2}$,
\[
g_{m,\theta}(\vx):=\frac{2^{1-m-d/2}}{\pi^{d/2}\Gamma(m)\theta^{2m-d}}(\theta\norm{\vx}_2)^{m-d/2}K_{d/2-m}(\theta\norm{\vx}_2),
\quad \vx\in\Rd,\quad\theta>0,
\]
is a full-space Green function of the differential operator $L:=(\theta^2I-\Delta)^m$, i.e., $Lg_{m,\theta}=\delta_0$, where $t\mapsto K_{\nu}(t)$ is the modified Bessel function of the second kind of order $\nu$. The kernel function
\[
K_{m,\theta}(\vx,\vy):=g_{m,\theta}(\vx-\vy),\quad \vx,\vy\in\Rd,
\]
is a positive definite kernel whose reproducing-kernel Hilbert space is equivalent to the $\Leb_2$-based Sobolev space of degree $m$, i.e., $\Hilbert_{K_{m,\theta}}(\Rd)\cong\Hil^m(\Rd)$.
Its inner product has the explicit form
\[
\langle f,g \rangle_{K_{m,\theta},\Rd}=\int_{\Rd}\vP_{\theta} f(\vx)^T\vP_{\theta} g(\vx)\ud\vx,\quad f,g\in\Hilbert_{K_{m,\theta}}(\Rd),
\]
where $\vP_{\theta}^T:=(\vQ^0,\vQ^1,\cdots,\vQ^m)$ and
\[
\vQ^j:=
\begin{cases}
a_j\Delta^k,&j=2k,\\
a_j\Delta^k\nabla^T,&j=2k+1,\\
\end{cases}
%\quad a_j:=\left(\frac{m!\theta^{2m-2j}}{j!(m-j)!}\right)^{1/2},
\quad a_j:=\sqrt{\frac{m!\theta^{2m-2j}}{j!(m-j)!}},
\quad k\in\NN_0,\quad j=0,1,\cdots,m.
\]
According to \cite[Theorem~1.4.6]{BerlinetThomas2004}, the reproducing-kernel Hilbert space $\Hilbert_{K_{m,\theta}}(\Domain)$ is endowed with the reproducing-kernel norm
\[
\norm{f}_{K_{m,\theta},\Domain}=\inf_{\tilde{f}\in\Hilbert_{K_{m,\theta}}(\Rd)}
\left\{\norm{\tilde{f}}_{K_{m,\theta},\Rd}:\tilde{f}|_{\Domain}=f\right\},\quad
f\in\Hilbert_{K_{m,\theta}}(\Domain).
\]
Moreover, if the open bounded domain $\Domain\subset\Rd$ is regular then $\Hilbert_{K_{m,\theta}}(\Domain)$ is equivalent to the $\Leb_2$-based Sobolev space of degree $m$, i.e.,
$
\Hilbert_{K_{m,\theta}}(\Domain)\cong\Hil^m(\Domain).
$

\end{example}
%////////////////////////////////////////////////////////////////////////////////////////////////////////////////////////

%---------------------------------------------------------------------------------------------------------------------
%/////////////////////////////////////////////////////////////////////////////////////////////////////////////////////
%---------------------------------------------------------------------------------------------------------------------

\section{Differential and Boundary Operators}\label{s:diff-bound}

In this section we define differential and boundary operators on Sobolev spaces
$\Hil^m(\Domain)$. The differential and boundary operators in this paper are well-defined since we assume that the open and bounded domain $\Domain$ is
\emph{regular}, i.e., it satisfies a strong local
Lipschitz condition which implies a uniform cone condition (see~\cite[Chapter~4.1]{AdamsFournier2003}).
This means that $\Domain$ has a regular boundary $\partial\Domain$.

Let the notation for typical derivatives be
\[
D^{\alpha}:=\prod_{k=1}^{d}\frac{\partial^{\alpha_k}}{\partial
x_k^{\alpha_k}}, \quad{}\abs{\alpha}:=\sum_{k=1}^d\alpha_k,\quad{}
\alpha:=\left(\alpha_1,\cdots,\alpha_d\right)\in\NN_0^d.
\]
We extend these derivatives to \emph{weak derivatives} (see~\cite[Chapter~1.5]{AdamsFournier2003}) using the same symbol $D^\alpha$. Using these weak derivatives, the \emph{classical $\Leb_2$-based Sobolev space} $\Hil^m(\Domain)$ is given by
\[
\Hil^m(\Domain):=\left\{f\in\Leb_1^{loc}(\Domain):D^{\alpha}f\in\Leb_2(\Domain),~\abs{\alpha}\leq
m,~\alpha\in\NN_0^d\right\},\quad{}m\in\NN_0,
\]
equipped with the natural inner product
\[
\langle f,g \rangle_{m,\Domain}:=\sum_{\abs{\alpha}\leq
m}\int_{\Domain}D^{\alpha}f(\vx)D^{\alpha}g(\vx)\ud\vx,\quad f,g\in\Hil^m(\Domain).
\]
The weak derivative $D^{\alpha}$ is a bounded linear operator from $\Hil^m(\Domain)$ into $\Leb_2(\Domain)$ when $\abs{\alpha}\leq m$.
Moreover, $D^{\alpha}f$ is well-posed on the boundary $\partial\Domain$ when $f\in\Cont^m(\overline{\Domain})$ and $\abs{\alpha}\leq m-1$ and we denote that $D^{\alpha}|_{\partial\Domain}f:=D^{\alpha}f|_{\partial\Domain}$. The book~\cite{AdamsFournier2003} and the paper~\cite{FasshauerYe2011} show that $D^{\alpha}|_{\partial\Domain}$ can be extended to a bounded linear operator from $\Hil^m(\Domain)$ into $\Leb_2(\partial\Domain)$ when $\abs{\alpha}\leq m-1$ because $\Domain$ has a regular boundary $\partial\Domain$.
The $\Leb_2(\partial\Domain)$-inner product is denoted by
\[
\langle f,g \rangle_{\partial\Domain}=\int_{\partial\Domain}f(\vx)g(\vx)\ud S(\vx),\quad
\text{when }d\geq2\text{ and }\partial\Domain\text{ is the boundary manifold},
\]
and
\[
\langle f,g \rangle_{\partial\Domain}=f(b)g(b)+f(a)g(a),\quad
\text{when }d=1\text{ and }\partial\Domain=\{a,b\}.
\]

%////////////////////////////////////////////////////////////////////////////////////////////////////////////////////////
\begin{definition}\label{d:diff-bound}
A \emph{differential operator}
$P:\Hil^m(\Domain)\rightarrow\Leb_2(\Domain)$ is well-defined by
\[
P=\sum_{\abs{\alpha}\leq m}c_{\alpha}
D^{\alpha},\quad{}\text{where
}c_{\alpha}\in\Cont^{\infty}(\overline{\Domain})\text{ and
}\alpha\in\NN_0^d,~m\in\NN_0,
\]
and its \emph{order} is given by $
\Order(P):=\max\left\{\abs{\alpha}:c_{\alpha}\not\equiv0,~\abs{\alpha}\leq
m,~\alpha\in\NN_0^d\right\}.
$
A \emph{boundary operator}
$B:\Hil^m(\Domain)\rightarrow\Leb_2(\partial\Domain)$ is well-defined
by
\[
B=\sum_{\abs{\alpha}\leq
m-1}b_{\alpha}D^{\alpha}|_{\partial\Domain},\quad{}\text{where
}b_{\alpha}\in\Cont^{\infty}(\partial\Domain)\text{ and
}\alpha\in\NN_0^d,~m\in\NN,
\]
and its \emph{order} is given by
$
\Order(B):=\max\left\{\abs{\alpha}:b_{\alpha}\not\equiv0,~\abs{\alpha}\leq
m-1,~\alpha\in\NN_0^d\right\}.$
\end{definition}
%////////////////////////////////////////////////////////////////////////////////////////////////////////////////////////
It is obvious that the differential operator $P$ and the boundary operator $B$ are bounded (continuous) linear operators on $\Hil^m(\Domain)$ with values in $\Leb_2$ whenever $\Order(P)\leq m$ and $\Order(B)\leq m-1$. Much more detail on differential and boundary operators can be found in~\cite{AdamsFournier2003,FasshauerYe2011}.

If $\Hilbert_K(\Domain)$ is embedded\footnote{The recent papers~\cite{FasshauerYe2011online,FasshauerYe2011,YeIITTech2010} show that there exist kernels $K$ whose reproducing-kernel Hilbert space $\Hilbert_K(\Domain)$ is continuously embedded into the Sobolev space $\Hil^m(\Domain)$.} into $\Hil^m(\Domain)$ then the eigenvalues $\{\lambda_k\}_{k=1}^{\infty}$ and eigenfunctions $\{e_k\}_{k=1}^{\infty}$ of the reproducing kernel $K$ satisfy
\[
\lambda_k\norm{e_k}_{m,\Domain}^2\leq C^2\norm{\sqrt{\lambda_k}e_k}_{K,\Domain}^2=C^2,\quad k\in\NN,
\]
because $\{\sqrt{\lambda_k}e_k\}_{k=1}^{\infty}$ is an orthonormal base of $\Hilbert_K(\Domain)$.
When $m>d/2$ then $\Hil^m(\Domain)$ is embedded into $\Cont(\overline{\Domain})$ by the Sobolev embedding theorem.
This implies that $K\in\Cont(\overline{\Domain}\times\overline{\Domain})\subset\Leb_2(\Domain\times\Domain)$ because $K(\cdot,\vy)\in\Cont(\overline{\Domain})$ for each $\vy\in\Domain$ and $K$ is symmetric.
Based on these properties, we can introduce the following lemma.
%////////////////////////////////////////////////////////////////////////////////////////////////////////////////////////
\begin{lemma}\label{l:P-B-expan}
Consider a differential operator $P$ with order $\Order(P)\leq m$ and a boundary operator $B$ with order $\Order(B)\leq m-1$, where $m>d/2$. If the reproducing-kernel Hilbert space $\Hilbert_K(\Domain)$ is embedded into the Sobolev space $\Hil^m(\Domain)$, then
\[
Pf=\sum_{k=1}^{\infty}\hat{f}_k\sqrt{\lambda_k}Pe_k,\quad
Bf=\sum_{k=1}^{\infty}\hat{f}_k\sqrt{\lambda_k}Be_k,\quad
f\in\Hilbert_K(\Domain),
\]
where $\hat{f}_k=\langle f,e_k \rangle_{K,\Domain}$ for each $k\in\NN$ and $\{\lambda_k\}_{k=1}^{\infty}$ and $\{e_k\}_{k=1}^{\infty}$ are the eigenvalues and eigenfunctions of the reproducing kernel $K$.
\end{lemma}
%////////////////////////////////////////////////////////////////////////////////////////////////////////////////////////
\begin{proof}
According to Theorem~\ref{t:RKHS} each $f\in\Hilbert_K(\Domain)$ can be expanded as $f=\sum_{k=1}^{\infty}\hat{f}_k\sqrt{\lambda_k}e_k$. Since $\{\sqrt{\lambda_k}e_k\}_{k=1}^{\infty}$ is an orthonormal basis we have $\sum_{k=1}^{\infty}\abs{\hat{f}_k}^2<\infty$.
Let $f_n:=\sum_{k=1}^{n}\hat{f}_k\sqrt{\lambda_k}e_k$ for each $n\in\NN$. Then
\[
\norm{f_n-f}_{m,\Domain}^2\leq C^2\norm{f_n-f}_{K,\Domain}^2
\leq C^2\sum_{k=n+1}^{\infty}\abs{\hat{f}_k}^2\rightarrow0,\quad \text{when }n\rightarrow\infty.
\]
The proof is completed by remembering that $P$ and $B$ are bounded linear operators on $\Hil^m(\Domain)$.
\end{proof}

If $\Hilbert_K(\Domain)$ is embedded into $\Hil^m(\Domain)$, then for each $\abs{\alpha}\leq m$, $\abs{\beta}\leq m$ and $\alpha,\beta\in\NN_0^d$, we have
\[
\begin{split}
%&\left(\int_{\Domain}\int_{\Domain}\abs{D^{\alpha}_{\vx}D^{\beta}_{\vy}K\ast K(\vx,\vy)}^2\ud\vx\ud\vy\right)^{1/2}=
&\left(\int_{\Domain}\int_{\Domain}\abs{\sum_{k=1}^{\infty}\lambda_k^2D^{\alpha}e_k(\vx)D^{\beta}e_k(\vy)}^2\ud\vx\ud\vy\right)^{1/2}\\
\leq&\sum_{k=1}^{\infty}\lambda_k^2\norm{D^{\alpha}e_k}_{\Leb_2(\Domain)}\norm{D^{\beta}e_k}_{\Leb_2(\Domain)}
\leq\sum_{k=1}^{\infty}\lambda_k^2\norm{e_k}_{m,\Domain}^2\leq C^2\sum_{k=1}^{\infty}\lambda_k<\infty,
\end{split}
\]
which implies that $\Kstar\in\Hil^{m,m}(\Domain\times\Domain)$.
Let $n:=\lceil m -d/2\rceil-1$. The Sobolev embedding theorem shows that $\Hil^{m,m}(\Domain\times\Domain)\subset\Cont^{n,n}(\overline{\Domain}\times\overline{\Domain})$.
Then we can obtain the following lemma.
%////////////////////////////////////////////////////////////////////////////////////////////////////////////////////////
\begin{lemma}\label{l:P-B-Cov-expan}
Consider a differential operator $P$ with order $\Order(P)< m-d/2$ and a boundary operator $B$ with order $\Order(B)< m-d/2$, where $m>d/2$. If the reproducing-kernel Hilbert space $\Hilbert_K(\Domain)$ is embedded into the Sobolev space $\Hil^m(\Domain)$, then
\[
\begin{split}
&P_1P_2\Kstar(\vx,\vy):=P_{\vz_1}P_{\vz_2}\Kstar(\vz_1,\vz_2)|_{\vz_1=\vx,\vz_2=\vy}=\sum_{k=1}^{\infty}\lambda_k^2Pe_k(\vx)Pe_k(\vy),\\
&B_1B_2\Kstar(\vx,\vy):=B_{\vz_1}B_{\vz_2}\Kstar(\vz_1,\vz_2)|_{\vz_1=\vx,\vz_2=\vy}=\sum_{k=1}^{\infty}\lambda_k^2Be_k(\vx)Be_k(\vy),\\
&P_1B_2\Kstar(\vx,\vy):=P_{\vz_1}B_{\vz_2}\Kstar(\vz_1,\vz_2)|_{\vz_1=\vx,\vz_2=\vy}=\sum_{k=1}^{\infty}\lambda_k^2Pe_k(\vx)Be_k(\vy),\\
\end{split}
\]
where $\{\lambda_k\}_{k=1}^{\infty}$ and $\{e_k\}_{k=1}^{\infty}$ are the eigenvalues and eigenfunctions of $K$.
Moreover, $P_1P_2\Kstar\in\Cont(\overline{\Domain}\times\overline{\Domain})$, $B_1B_2\Kstar\in\Cont(\partial\Domain\times\partial\Domain)$
and $P_1B_2\Kstar\in\Cont(\overline{\Domain}\times\partial\Domain)$.
\end{lemma}
%////////////////////////////////////////////////////////////////////////////////////////////////////////////////////////

%------------------------------------------------------------------------------------------------------------------------
%------------------------------------------------------------------------------------------------------------------------

\section*{Acknowledgments}

We would like to thank Phaedon-Stelios Koutsourelakis for the inspiration to solve SPDEs with a maximum likelihood-based approach.
The work of Igor Cialenco was partially supported by the National Science Foundation (NSF) grant DMS-0908099.
Gregory E. Fasshauer and Qi Ye acknowledge support from NSF grants DMS-0713848 and DMS-1115392.
The authors would like to thank the anonymous referee and the editors for their helpful comments and suggestions which improved greatly
the final manuscript.

%\bibliographystyle{plain}
%\bibliography{QiYeJabRef3}

\begin{thebibliography}{9}

\bibitem{AdamsFournier2003}
R.~A. Adams and J.~J.~F. Fournier,
\emph{{S}obolev {S}paces}.
vol. 140 of pure and applied mathematics,
Elservier/Academic Press, Amsterdam, 2003.


\bibitem{BabuvskaNobileTempone2010}
I. Babu{\v{s}}ka, F. Nobile and R. Tempone,
\emph{{A} {S}tochastic {C}ollocation {M}ethod for {E}lliptic {P}artial {D}ifferential {E}quations with {R}andom {I}nput {D}ata},
SIAM Rev., vol.~52, 2010, pp.~317--355.


\bibitem{BerlinetThomas2004}
A. Berlinet and C. Thomas-Agnan,
\emph{{R}eproducing {K}ernel {H}ilbert {S}paces in {P}robability and {S}tatistics},
Kluwer Academic Publishers, 2004.


\bibitem{Buhmann2003}
M. D. Buhmann,
\emph{{R}adial {B}asis {F}unctions: {T}heory and {I}mplementations}.
vol. 12 of Cambridge monographs on applied and computational mathematics,
Cambridge University Press, 2003.


\bibitem{Chow2007}
P.-L. Chow,
\emph{{S}tochastic {P}artial {D}ifferential {E}quations}.
Applied Mathematics and Nonlinear Science Series, Chapman \& Hall/CRC, Boca Raton, FL, 2007.


\bibitem{DaPrato}
G. Da Prato and J. Zabczyk,
\emph{{S}tochastic {E}quations in {I}nfinite {D}imensions}.
vol. 44 of Encyclopedia of Mathematics and its Applications, 
Cambridge University Press, 1992.


\bibitem{DebBabuvskaOden2001}
M. K. Deb, I. M. Babu{\v{s}}ka, and J. T. Oden,
\emph{{S}olution of stochastic partial differential equations using {G}alerkin finite element techniques}.
Comput. Methods Appl. Mech. Engrg., vol. 190, no. 48, 2001, pp.~6359--6372. 



\bibitem{Fasshauer2007}
G.~E. Fasshauer,
\emph{{M}eshfree {A}pproximation {M}ethods with {\sc{Matlab}}},
World Scientific Publishing Co. Pte. Ltd., 2007.


\bibitem{Fasshauer2011}
G.~E. Fasshauer,
\emph{{P}ositive definite kenrels: past, present and future}.
Dolomite Research Notes on Approximation, vol.~4, 2011, pp.~21--63.


\bibitem{FasshauerYe2011}
G.~E. Fasshauer and Q. Ye,
\emph{{R}eproducing kernels of generalized {S}obolev spaces via a {G}reen function approach with distributional operators}, Numer. Math., vol.~119, 2011, pp.~585--611.


\bibitem{FasshauerYe2011online}
G.~E. Fasshauer and Q. Ye,
\emph{{R}eproducing kernels of {S}obolev spaces via a {G}reen kernel approach with differential operators and boundary operators}, Adv. Comput. Math., DOI: 10.1007/s10444-011-9264-6.


\bibitem{HonSchaback2010}
Y. C. Hon and Robert Schaback,
\emph{{T}he kernel-based method of lines for the heat equation}.
University of G\"ottingen, preprint, 2010,
\url{http://num.math.uni-goettingen.de/schaback/research/papers/MKTftHE.pdf}.



\bibitem{Janson1997}
S. Janson,
\emph{{G}aussian {H}ilbert {S}paces}.
vol. 129 of Cambridge Tracts in Mathematics, 
Cambridge University Press, 1997.


\bibitem{JentzenKloeden2010a}
A. Jentzen and P. Kloeden,
\emph{Taylor expansions of solutions of stochastic partial differential equations with additive noise}.
Ann. Probab., vol. 38, no. 2, 2010, pp.~532--569.


\bibitem{KaratzasShreve1991}
I. Karatzas and S. E. Shreve,
\emph{{B}rownian {M}otion and {S}tochastic {C}alculus}.
vol. 113 of Graduate Texts in Mathematics, New York, 1991.


\bibitem{LukicBeder2001}
M.~N. Luki{\'c} and J.~H. Beder,
\emph{{S}tochastic {P}rocesses with {S}ample {P}aths in {R}eproducing {K}ernel {H}ilbert {S}paces}.
Trans. Amer. Math. Soc., vol.~353, 2001, pp.~3945--3969.


\bibitem{Muller-GronbachRitterWagner2008}
T. M{\"u}ller-Gronbach, K. Ritter and T. Wagner,
\emph{{O}ptimal pointwise approximation of a linear stochastic heat equation with additive space-time white noise}.
Monte {C}arlo and {Q}uasi-{M}onte {C}arlo {M}ethods 2006, Berlin, 2008, pp.~577--589.


\bibitem{NobileTemponeWebster2008}
F. Nobile, R. Tempone and C. G. Webster,
\emph{{A} sparse grid stochastic collocation method for partial differential equations with random input data}.
SIAM J. Numer. Anal., vol. 46, no. 5, 2008, pp.~2309--2345.


\bibitem{RozovskiiBook}
B. L. Rozovskii,
\emph{{S}tochastic {E}volution {S}ystems: {L}inear {T}heory and {A}pplications to {N}onlinear {F}iltering}.
vol. 35 of Mathematics and its Applications (Soviet Series), Kluwer Academic Publishers Group, Dordrecht, 1990.


\bibitem{ScheuererSchabackSchlather2010}
M. Scheuerer, R. Schaback and M. Schlather,
\emph{{I}nterpolation of spatial data -- a stochastic or a deterministic problem?}.
Data Page of R. Schaback's Research Group, 2010, 
\url{http://num.math.uni-goettingen.de/schaback/research/papers/IoSD.pdf}.



\bibitem{Wendland2005}
H. Wendland,
\emph{{S}cattered {D}ata {A}pproximation},
Cambridge University Press, 2005.


\bibitem{YeIITTech2010}
Q. Ye,
\emph{{R}eproducing kernels of generalized {S}obolev spaces via a {G}reen function approach with differential operators}.
Illinois Institute of Technology, 2010, \url{arXiv:1109.0109v1}.



\bibitem{YePhD2012}
Q. Ye,
\emph{{A}nalyzing reproducing kernel approximation methods via a {G}reen function approach}.
Ph.D. thesis, Illinois Institute of Technology, 2012.


\end{thebibliography}

%---------------------------------------------------------------------------------------------------------------------
%/////////////////////////////////////////////////////////////////////////////////////////////////////////////////////
%---------------------------------------------------------------------------------------------------------------------

\end{document}